\documentclass[11pt]{amsart}   

\usepackage{amsmath}
\usepackage{amsfonts}
\usepackage{amssymb}
\usepackage{stmaryrd} 
\usepackage[headings]{fullpage}   
\usepackage[all]{xy}
\usepackage{verbatim}
\usepackage{color}
\usepackage[usenames,dvipsnames,svgnames,table]{xcolor}

\usepackage[breaklinks,bookmarksopen,bookmarksnumbered,urlcolor=blue]{hyperref}
\hypersetup{colorlinks=true,backref=true,citecolor=blue}

\usepackage{tikz,color}   
\usetikzlibrary{arrows}  
\usetikzlibrary{decorations.pathmorphing}  

 \usepackage{subdepth} 
  \usepackage{graphicx} 
 
\usepackage[breaklinks,bookmarksopen,bookmarksnumbered]{hyperref}
\hypersetup{colorlinks=true,backref=true,pagebackref=true,citecolor=blue}

 \makeatletter \def\theenumi{\@roman\c@enumi}  \def\@linkcolor{black} \def\@citecolor{black}  \makeatother  

\makeatletter
\font \sectionfont = cmbx12 
\renewcommand\section{%\vfil\break
\@startsection{section}%
{1}{0pt}{-1.8\baselineskip}{.8\baselineskip}%
{\sectionfont \raggedright}}
\makeatother

\def\mthmskip{\vskip .2 true cm}

\theoremstyle{plain}
\newtheorem{thm}{Theorem}[section]
\newtheorem{lemma}[thm]{Lemma}
\newtheorem{cor}[thm]{Corollary} 
\newtheorem{prop}[thm]{Proposition}
\newtheorem{exercise}[thm]{Exercise}
\newtheorem{claim}[thm]{Claim}
\newtheorem{conjecture}[thm]{Conjecture}
\newtheorem{question}[thm]{Question}
\newtheorem{problem}[thm]{Problem}
\newtheorem{openpr}[thm]{Open-Ended Problem}

\theoremstyle{definition}
\newtheorem{definition}[thm]{Definition}
\newtheorem{example}[thm]{Example}

\newtheorem{notation}[thm]{Notation}

\newtheorem{assumptions and notation}[thm]{Assumptions and Notation}

\newtheorem{remark}[thm]{Remark}

\newtheorem{observation}[thm]{Observation}
\newtheorem{construction}[thm]{Construction}

\def\bprop{\begin{prop}}\def\eprop{\end{prop}}
\def\bthm{\begin{thm}}\def\ethm{\end{thm}}
\def\bcor{\begin{cor}}\def\ecor{\end{cor}}
\def\blemma{\begin{lemma}}\def\elemma{\end{lemma}}
\def\bexercise{\begin{exercise}}\def\eexercise{\end{exercise}}
\def\bconjecture{\begin{conjecture}}\def\econjecture{\end{conjecture}}
\def\bclaim{\begin{claim}}\def\eclaim{\end{claim}}
\def\bquestion{\begin{question}}\def\equestion{\end{question}}
\def\bproblem{\begin{problem}}\def\eproblem{\end{problem}}
\def\boproblem{\begin{openpr}}\def\eoproblem{\end{openpr}}

\def\bconstruction{\begin{construction}}\def\econstruction{\end{construction}}
\def\bdefinition{\begin{definition}}\def\edefinition{\end{definition}}
\def\bremark{\begin{remark}}\def\eremark{\end{remark}}
\def\bnotation{\begin{notation}}\def\enotation{\end{notation}}
\def\bobservation{\begin{observation}}\def\eobservation{\end{observation}}
\def\bexample{\begin{example}}\def\eexample{\end{example}}

\def\bequation{\addtocounter{thm}{1}\begin{equation}} 
\def\eequation{\end{equation}}

\def\balign{\begin{align*}}
\def\ealign{\end{align*}} 

\def\bnualign{\addtocounter{thm}{1}\begin{align}}
\def\enualign{\end{align}}

\def\myalign{\addtocounter{thm}{1}\begin{align}}

\def\citem#1{\item[{\rm (#1)}]} 

\newenvironment{bmlist}    
{\begin{list}{\text} 
{\setlength{\rightmargin}{0cm}\setlength{\itemindent}{0cm}
\setlength{\labelwidth}{1cm}\setlength{\leftmargin}{1.2cm}
\setlength{\labelsep}{.2cm}}}{\end{list}} 

\def\emlist{\end{list}}

\def\bmatrix{\begin{pmatrix}}
\def\ematrix{\end{pmatrix}}

\def\to{\longrightarrow} 
\def\map#1{\xrightarrow{\ #1\ }} 

\def\cal{\mathcal}

\font \prooffont = cmcsc10 
 \font \sectionfont = cmbx12

\font \prooffont = cmcsc10 

\def\mproof{\vskip 0mm\noindent {\prooffont Proof:\ \ }} 

\def \mproofof#1#2{\vskip 0mm\noindent {\it Proof of #1 {\rm #2}.\ \ }} 

\def\bproof{\begin{proof}} 
\def\eproof{\end{proof}}

\def\mqed{\unskip\nobreak\hfil\penalty50\hskip1em\null\nobreak\hfil\mqedsymbol
\parfillskip=0pt\finalhyphendemerits=0\endgraf}

\def\mqed{\unskip\nobreak\hfil\penalty100\hfill\null\nobreak\hfil
        {\vbox{\hrule\hbox{\vrule\kern3.1pt
        \vbox{\kern6pt}\kern3pt\vrule}\hrule}}  {\parfillskip=0pt\finalhyphendemerits=0\par} 
        \mthmskip }
     
\def\b{\begin}  
\def\l{\label}

\def \mspace{\vglue .2cm}

 \newtheoremstyle{figstyle}
  {} % 
  {.5cm} % 
    {} % 
  {4.8cm} %
  {\bfseries} % 
  {.} % 
  {.5em} % 
  {} % 

\def\bfigure#1{
\vglue .5cm\begin{center} \begin{tikzpicture}
 [every node/.style={scale=#1},  auto]}
 \def\efigure{\end{tikzpicture} \end{center}\vglue .3cm}

\def\bdiagram#1{
\vglue .2cm\begin{center} \begin{tikzpicture}
 [every node/.style={scale=#1},  auto]}
 \def\ediagram{\end{tikzpicture} \end{center}\noindent}

\def\dim{{\rm{dim}}}

\def\ext{{\rm{Ext}}}

\def\Ho{{\rm H}}     %
\def\hom{{\rm{Hom}}}

\def\rank{{\rm{rank}\,}}
\def\reg{{\rm{reg}}}

\def\sup{{\rm sup}}

\def\tor{{\rm{Tor}}}

\def\B{{\bf B}}

\def\K{{\bf K}}
\def\T{{\bf T}}

\def\F{{\bf F}}

\def\G{{\bf G}}

\def\L{{\bf L}}

\def\to{\longrightarrow}

\def\RR{{\mathcal R}}

\def\mm{{\mathfrak m}}
\def\gm{{\mathfrak m}}
\def\ff{{f_{1}, \dots, f_{c}}}
\def\ff#1{{f_{1},\dots, f_{#1}}}

\def\Tor{{\rm{Tor}}}

\def\B{{\bf B}}

\def\K{{\bf K}}
\def\T{{\bf T}}

\def\F{{\bf F}}

\def\G{{\bf G}}

\def\L{{\bf L}}

\def\dim{{\rm{dim}}}
\def\ext{{\rm{Ext}}}

\def\Ho{{\rm H}} % 
\def\hom{{\rm{Hom}}}

\def\rank{{\rm{rank}\,}}
\def\reg{{\rm{reg}}}
\def\sup{{\rm{sup}}}

\def\tor{{\rm{Tor}}}

\def\to{\longrightarrow}

\def\RR{{\mathcal R}}

\def\ff{{f_{1}, \dots, f_{c}}}
\def\ff#1{{f_{1},\dots, f_{#1}}}

\def\sym{{\rm Sym}}

\def\CR{{\mathcal R}}

\def\mm{{\mathfrak m}}
\def\even{{\rm even}}
\def\odd{{\rm odd}}

 \def\ldown{{\bf L\kern-2.5pt\downarrow}}
  \def\lup{{\bf L\kern-2.5pt\uparrow}}

  \def\f{{\bf f}}

\begin{document}

\baselineskip 14.5  pt  %

\title{Tor as a Module over an Exterior Algebra}
\author[David Eisenbud]{David Eisenbud}
\address{Mathematics Department, University of California at Berkeley, Berkeley, CA 94720, USA}
\email{de@msri.org}

\author[Irena Peeva]{Irena Peeva}
\address{Mathematics Department, Cornell University, Ithaca, NY 14853, USA}

\author[Frank-Olaf Schreyer]{Frank-Olaf Schreyer}
\address{Fachbereich Mathematik, Universit\"at des Saarlandes, Campus E2 4, D-66123 Saar\-br\"ucken, Germany}
\email{schreyer@math.uni-sb.de}

\begin{abstract}
Let $S$ be a regular local ring with residue field $k$ and let $M$ be a finitely
generated $S$-module. Suppose that
 $f_1,\dots ,f_c\in S$ is a regular sequence that annihilates $M$,
 and let $E$ be an exterior algebra over $k$ generated by $c$ elements.
 The homotopies for  the $f_{i}$ on a free
resolution of $M$ induce  a natural structure of graded $E$-module  on 
 $\tor^{S}(M,k)$.  In the case where
 $M$ is a high syzygy over the complete intersection $R:=S/(f_{1},\dots,f_{c})$
 we describe this $E$-module structure in detail, including
 its minimal free resolution over $E$.
 
Turning to 
 $\ext_{R}(M,\, k)$ we
show that, when $M$ is a high syzygy over $R$, the
 minimal free resolution of $\ext_{R}(M,\, k)$ as a module over the ring
 of CI operators is the Bernstein-Gel'fand-Gel'fand dual of the $E$-module
$\tor^{S}(M,\,k)$. 

For the proof we introduce \emph{higher CI operators}, and
give a construction of a (generally non-minimal) resolution of $M$ over
$S$ starting from a resolution of $M$ over $R$ and its higher CI operators.

 \end{abstract}

\thanks{The work on this paper profited from the good conditions for mathematics at MSRI, and was partially supported by the National Science Foundation under Grant 0932078000.  The first two authors are grateful to the 
National Science Foundation for partial support under Grants DMS-1502190, DMS-1702125 and  DMS-1406062.}



\maketitle

\section{Introduction} \label{Sintro}
 Throughout this paper we write $S$ for a regular local ring with maximal ideal $\gm$ and residue field $k$, and we let
 $f_1,\dots ,f_c\in S$ be a regular sequence. 
 Set $I := (f_{1},\dots,f_{c})\subset S$ and consider the complete intersection $R:= S/I$.
 Let
$M$ be a finitely generated
 $S$-module
annihilated by
 $I$.   We denote by $E$ the exterior algebra
 $$E:= \wedge_{k}(I/\gm I) =: k\langle e_{1}, \dots e_{c}\rangle\,.$$
 The finite-dimensional graded vector space $\tor^{S}(M,k)$
has a natural $E$-module structure induced by the action of  homotopies for the $f_i$ on the minimal $S$-free resolution of $M$ (Section~\ref{exterior structure}).  
For some modules $M$, the action of $E$ on $\tor_{S}(M,k)$ is trivial but,
 in the case where
 $M$ is a high $R$-syzygy  in the sense of \cite{EP} (explicit bounds  are given in \cite{EP} and \cite{EP1}) we prove that it is highly nontrivial:
\bmlist
 
\citem{i}  We prove that the $E$-module $\tor^{S}(M,k)$ is generated  
 by $\tor^{S}_{0}(M,k)$ and $\tor^{S}_{1}(M,k)$, and its (Castelnuovo-Mumford) regularity  is $1$ (Corollary~\ref{torreg} and Theorem~\ref{reg1}).

\citem{ii} 
 Let
  $$
 T' := E\cdot \tor^{S}_{0}(M,k) \subset \tor^{S}(M,k)
  $$
and let  
$$T'' : =  \tor^{S}(M,k)/T'$$
be the quotient.
Assuming that $k$ is infinite and the generators of $(\ff c)$ are chosen generally, we compute vector space bases of $T'$ and $T''$, and show that, as $E$-modules, $T'$ and $T''$ have Gr\"obner deformations to  direct sums of copies of $E/(e_{p}, \dots, e_{c})$ for $p = 1,\dots ,c$ (Theorem~\ref{twores}). It follows that, even when $k$ is finite, $T'$ and $T''$ have linear $E$-free resolutions, given explicitly in (iv) below. 

\citem{iii} We prove that the Betti numbers of the $0$-linear strand 
of the  minimal  $E$-free graded resolution of  $\tor^{S}(M,k)$
 are given by the even Betti numbers of $M$ over $R$, and the  Betti numbers of the $1$-linear strand are given by the odd Betti numbers of $M$ over $R$. That is:
\b{align*}
\beta_{i,i}^E\Big( \tor^{S}(M,k) \Big)&=\beta_{2i}^R(M)
\\
\beta_{i,i+1}^E\Big( \tor^{S}(M,k) \Big)& =\beta_{2i+1}^R(M)
\, .
\end{align*}
(Theorem \ref{resTor}).

\citem{iv} We show that the numerical statement in (iii) is a consequence of the
structure of the  minimal $E$-free resolution of $\tor^S(M,k)$ by proving that 
the resolution  is the mapping
cone:
%\vglue -.5cm
 \begin{center} 
\begin{tikzpicture} [scale = 0.8, everynode/.style = {scale = 0.9}]
\node (DE)  {$\dots$}; 
\node (4)[node distance=2.7cm, right of=DE]       {$\Tor^{R}_{4}(M,k)\otimes_{R} E$}; 
\node (2)   [node distance=4cm, right of=4]      {$ \Tor^{R}_{2}(M,k)\otimes_{R} E$}; 
\node (0)  [node distance=4cm, right of=2]       {$\Tor^{R}_{0}(M,k)\otimes_{R}E$}; 

\node (OD)  [node distance=1cm, below of=DE]   {$\dots$}; 
\node (O1)    [node distance=1cm, below of=4]      {$\bigoplus$}; 
\node (O2)    [node distance=1cm, below of=2]      {$\bigoplus$}; 
\node (O3)     [node distance=1cm, below of=0]     {$\bigoplus$};  

\node (DO)  [node distance=1cm, below of=OD]  {$\dots$}; 
\node (5)    [node distance=1cm, below of=O1]      {$\Tor^{R}_{5}(M,k)\otimes_{R}E$}; 
\node (3)    [node distance=1cm, below of=O2]      {$\Tor^{R}_{3}(M,k)\otimes_{R}E$}; 
\node (1)      [node distance=1cm, below of=O3]     {$\Tor^{R}_{1}(M,k)\otimes_{R}E\,,$}; 

\draw[->] (DE) to node [above=10pt, right=-8pt]  {$t_{2} $} (4);
\draw[->] (4) to node [above=10pt, right=-8pt]  {$t_{2} $} (2);
\draw[->] (2) to node [above=10pt, right=-8pt]  {$t_{2} $} (0);

\draw[->] (DO) to node [above=10pt, right=-8pt]  {$t_{2} $} (5);
\draw[->] (5) to node [above=10pt, right=-8pt]  {$t_{2} $} (3);
\draw[->] (3) to node [above=10pt, right=-8pt]  {$t_{2} $} (1);

\draw[->] (5) to node [above=6pt, right=-15pt]  {$t_{3} $} (2);
\draw[->] (3) to node [above=6pt, right=-15pt]  {$t_{3} $} (0);

\end{tikzpicture} \end{center}
(Theorem~\ref{main7}, see also Theorem~\ref{twores} (iii)) where the two rows are themselves minimal linear free resolutions of the 
$E$-submodule $T'$ 
 and
the quotient $T''$. 
The maps labeled  $t_{2}$ are the CI (=Complete Intersection) operators (also called Eisenbud operators), while the maps labeled  $t_3$ between the two strands are some of the higher CI-operators, introduced in Section~\ref{inverseShamash}.  

\end{list}

\noindent 
A curious consequence of (iv) is that if $M$ is a high syzygy over $R$, then the   \emph{first} syzygy over $E$ of the $E$-module $\tor^{S}(M,k)$ is $\tor^{S}(L,k)$, where $L$ is the \emph{second} syzygy of $M$ as an $R$-module. For a slightly sharper statement see Corollary~\ref{first syzygy}.

\vglue .25cm
Next we focus on $\ext_{R}(M,k)$.  
The action of the CI operators makes the graded vector space $\ext_{R}(M,k)$ into a finitely generated module
over the ring $$ \CR:= \sym_{k}\left((I/\gm I)^\vee\right) =: k[\chi_{1},\dots,\chi_{c}]\,.$$ 
In Theorem~\ref{res of ext} we prove  that
when $M$ is a high $R$-syzygy,  the minimal $\CR$-free resolution of 
$\ext_R^{\even}(M,k)$ is obtained by the Bernstein-Gel'fand-Gel'fand (BGG) correspondence from the $E$-module structure of $T'^\vee$, and similarly for
$\ext_R^{\odd}(M,k)$ and $T''^\vee$.

Corollary~\ref{nonminimal presentation} doesn't even require the definition of $T'$. Write 
$$
\mu: E_{1}\otimes_{k}\tor^{R}_{0}(M,k) \to  \tor^{R}_{1}(M,k)
$$
 for the multiplication map
and 
$$
\mu^{\vee}: \ext_{S}^{1}(M,k)\to  \ext_{S}^{1}(M,k) \otimes \CR_{1}
$$
for its vector space dual.
The $\CR$-module $\ext_{R}^{\even}(M,k)$ then has  the (non-minimal) linear free presentation 
$$
\ext_{S}^{1}(M,k)\otimes \CR(-1) \map{\tau}  \hom(M,k) \otimes \CR \to \ext_{R}^{\even}(M,k)\to 0
$$
where $\tau$ is the map of free modules whose linear part is $\mu^\vee$. This follows from  Theorem~\ref{res of ext} because $\mu^\vee$ is 0 on the submodule $T''^\vee$.

An essential ingredient in the proofs in Section~\ref{Identifying the Ext Module} is a new theory of \emph{higher CI operators}, introduced in Section~\ref{inverseShamash}.  
Just as the Eisenbud-Shamash construction allows one to describe an $R$-free resolution of any $R$-module from the higher homotopies on an $S$-free resolution, one can describe an $S$-free resolution from the higher CI-operators on an $R$-free resolution. This construction was discovered independently by Jessie Burke~\cite{Burke}.  The  differentials in the $E$-free resolution of $\tor_S(M,k)$ are related, as above, to the higher CI-operators.

We also use
the ``layered" structures of the minimal $S$-free and $R$-free resolutions of $M$ \cite{EP1}, which come from the higher matrix factorizations of~\cite{EP}.
We  review the  necessary definitions and results about layered resolutions in Section~\ref{layered res}.

\vskip.1cm
\noindent{\bf Remark.} 
One could often replace the hypothesis that $S$ is regular with a hypothesis that $S$ is Gorenstein and $M$ has finite projective dimension over $S$, or  that $M$ is the module of a minimal higher matrix factorization. Moreover, the hypothesis that $M$ is the module of a minimal higher matrix factorization could be replaced  by the (possibly) more general hypothesis that the layered resolutions described in Section~\ref{layered res} are minimal. We leave these refinements to the interested reader. 
\vskip.1cm
\goodbreak

 The following example is computed using 
Macaulay2 code, which 
may be found in the documentation for the function ``exteriorTorModule" in the
Macaulay2 package CompleteIntersectionResolutions.m2 \cite[Version 1.9.1 and higher]{M2}.

\begin{example}\l{introex}
Let $S = k\llbracket x_1,x_2, x_{3}\rrbracket $, and $R = S/(x_{1}^{3}, x_{2}^{3}, x_{3}^{3})$. Denote by
$N_{i}$  the $i$-th syzygy of $k$ as an $R$-module.
 The minimal $S$-free resolution of $N_0=k $
is the Koszul complex on  $x_1,x_2, x_3$, and 
$x_{i}^{2}$ times the homotopy for $x_{i}$ is a homotopy $\sigma_{i}$ for $f_{i} := x_i^3$. Thus the action of $E$ on $\tor^{S}(k,k)$ is trivial.

By contrast, the action of $E$ on $\tor^{S}(N_{i},k)$ is nontrivial for $i\geq 1$. The beginnings of the Betti tables of the minimal $E$-free resolutions of $\tor^{S}(N_{i},k)$ for $i = 1,2,3$ are:
\begin{align*}
\hbox{Betti table of }\tor^{S}(N_{1},k): \ & \quad
\begin{matrix}
{\rm total:}& 10&27&52&85&126&175\cdots\\
\hline
0:&3&9&18&30&45&63&\cdots\\
1:&6&15&28&45&66&91&\cdots\\
2:&1&3&6&10&15&21&\cdots
\end{matrix}
\\ \\
\hbox{Betti table of }\tor^{S}(N_{2},k): \ & \quad
\begin{matrix} 
{\rm total:}& 16 &36& 64& 100& 144& 196&\cdots\\
\hline
0:&6&15&28&45&66&91&\cdots\\
1:&10& 21 &36 & 55 & 78 &105&\cdots
\end{matrix}
\\ \\
\hbox{Betti table of }\tor^{S}(N_{3},k): \ & \quad
\begin{matrix} 
{\rm total:}& 25 &49 &81 &121 &169 &225&\cdots\\
\hline
0:&10& 21 &36 & 55 & 78 &105&\cdots\\
1:&15 & 28& 45&  66&  91 &120&\cdots
\end{matrix}
\end{align*}

From the first table we see that, as an $E$-module, $\tor^{S}(N_{1},k)$ is
generated in degrees 0,1,2. Since $N_1$ is artinian, we have  $\tor_{3}^{S}(N_{1},k)\neq 0$. Hence, the $E$-module structure of $\tor^{S}(N_{1},k)$ is nontrivial.

The smallest $i$ for which $N_{i}$ is a high syzygy (in the sense of \cite{EP}) is $i=3$, but in fact the \emph{layered resolution} of $N_2$  with respect to $f_1, f_2, f_3$ described in Section~\ref{layered res} is also minimal. The $E$-module
 $\tor^{S}(N_{i},k)$ has a free resolution with just 2 linear strands for $i= 2,3$, illustrating assertion (i) above.
 
Further, Macaulay2 computes the Betti table of $N_{2}$ as an $R$-module
as 
$$
\begin{matrix}
{\rm total:} & 6 & 10 &15 &21 &28 &36 &45 &55 &66 &78 &91&105& \cdots\\
\hline
0:& 6 & 10 &15 &21 &28 &36 &45 &55 &66 &78 &91&105& \cdots
\end{matrix}
$$
and we see that, for $s = 0,1$,
$$
 \beta^{E}_{i,i+s}\Big(\tor^{S}(N_{2},k)\Big)=\beta^{R}_{2i+s}(N_{2}) \,,
$$
which illustrates (iii).

We can also illustrate Theorem~\ref{res of ext} in this context. It turns out that the  homogeneous components
of $T' = E \cdot\Tor_0^S(N_2,k)$ are $T'_0 = E^6,\ T'_1 = E^3,\ T'_2 = E^1$, and
it follows that
the minimal $\CR$-free resolution of $\ext_R^{\rm even}(N_2,k)$  has the form
$$
0\to \RR^1(-2) \map{d_2}\RR^3(-1) \map{d_1}\RR^6
\to\ext_R^{\rm even}(N_2,k)\to 0.
$$
 The differentials are easily computed from the action of $E$ on $\tor^S(N_2,k)$. The map 
 $$
\langle \ff 3\rangle \otimes \tor_0^S(N_2,k) \to \tor_1^S(N_2,k),
$$ 
given by the homotopies induces the map $E_1\otimes T'_0 \to T'_1$, whose
dual induces
 the map $d_1:\,\RR^3(-1) \to \RR^6$. Computation
shows that, in suitable bases, this map is given by the matrix
$$
\begin{pmatrix}
0&0&0\\
0&0&0\\
0&0&0\\
0&\chi_1&\chi_2\\
-\chi_1&0&\chi_3\\
-\chi_2&-\chi_3&0\\
\end{pmatrix}
$$
where the $\chi_i$ form a dual basis to $f_1,f_2,f_3$. From this presentation matrix we see that $\ext_R^{\rm even}(N_2,k)$ is the direct
sum of $\RR^3$ and one copy of the maximal ideal 
$(\chi_1,\chi_2,\chi_3)\subset \RR$, shifted so that
it is generated in degree 0.
Similar conclusions hold for $\ext^{\rm odd}(N_2,k)$.
\end{example}

\noindent{\bf Related Work.} Avramov and Buchweitz made  use of the simple classification of modules 
over an exterior algebra on 2 generators to study free resolutions of modules over complete intersections of codimension 2 in \cite{AB}, and this study is carried further in \cite{AZ}. For other points of view on the module structure of $\tor$ see \cite{Da, HW}. For further results on  resolutions over exterior algebras, see for example
\cite{AI, Ei2, Fl}. 

\vglue .25cm
\noindent{\bf Acknowledgements.}
Computations with Macaulay2~\cite{M2}  led us to guess the statements of our main theorems.
Many of the constructions in this paper are coded in the packages 
\href{http://www.math.uiuc.edu/Macaulay2/doc/Macaulay2-1.8.2/share/doc/Macaulay2/BGG/html/}{BGG.m2}
 and 
 \href{http://www.math.uiuc.edu/Macaulay2/doc/Macaulay2-1.8.2/share/doc/Macaulay2/CompleteIntersectionResolutions/html/}
{CompleteIntersectionResolutions.m2} distributed with the
Macaulay2 system. We want to express our gratitude to 
the authors Dan Grayson and  Mike Stillman of Macaulay2 for their unfailing patience in answering our questions about the program.

\section{Homotopies and the action of the exterior algebra on $\tor$}\label{exterior structure}\l{general tor module}

In this section we review  the action of the exterior algebra on Tor.
We will use the notation at the beginning of the Introduction.

For each $i$ we choose  a homotopy $\sigma_{i}$ for $f_{i}$ on a free resolution
$\F$ of the module $M$. 
Up to homotopy, the homotopies  $\sigma_{i}$ anticommute and square to 0---see for example \cite[Proposition 3.4.2]{EP}. Though the $\sigma_i$ are not maps of complexes,
they become maps of complexes when tensored with an $S$-module $N$
annihilated by $\ff c$, so $\sigma_{i}\otimes 1$ takes cycles in the complex
${\bf F} \otimes_{S}N$ to cycles, while raising the homological degree by 1.  Thus
the action of the $\sigma_{i}$
gives $H_{*}({\bf F}\otimes N) = \tor^{S}(M,N)$ the structure of a graded
module over the exterior algebra $\wedge_{S}(I)$. The action factors through an action of
$\wedge_S(I/I^2)$ because $I\cdot \tor^S(M,N) = 0$. 

As an example, consider the Koszul complex $\K =\K(\ff c)$. 
Denote by $e_{i}$ the basis element  of $\K_1$ that maps to $f_i$.  An immediate computation shows that multiplication by $e_i$ is a homotopy
for $f_{i}$. These homotopies anti-commute and square to 0, making
$\K$  a free module over the exterior algebra $S \otimes_{k}\wedge_{k}k^{c} = S\langle e_{1}, \dots, e_{c}\rangle$. 

The action of $\wedge(I/I^2)$ on $\tor^S(M,N)$ is functorial by  Lemma~\ref{functorial}.

\begin{lemma}\label{functorial}
Let 
$(\F,\partial): \ \cdots\to F_1\to F_0$ 
and $(\G,d):\  \cdots\to G_1\to G_0$ be complexes of  $S$ modules, 
and suppose that $\F$ and $\G$ admit homotopies $\sigma$
 for some element $f\in S$, respectively. Let $\varphi: \F\to\G$ be a map of complexes. If the $F_i$ are
free and the complex $\G$ is acyclic, then there are maps $\alpha_i: F_i\to G_{i+2}$
such that 
$\varphi\sigma-\tau\varphi = d\alpha-\alpha\partial$.

Thus,
if $N$ is  a module
annihilated by $f$, 
 the maps $\varphi\sigma$ and $\tau\varphi$ are homotopic maps of complexes; in particular,
 the maps 
\balign
\Ho(\F\otimes N)& \to \Ho(\G\otimes N)\\
\Ho(\hom(\G,N)) & \to \Ho(\hom(\F,N))
\end{align*}
induced by $\varphi$ commute with the action of the homotopies.
\end{lemma}

The proof is an immediate computation.

For example, taking $\varphi$ to be the identity on a free resolution of $M$, we see that the induced action of $\wedge_S(I/I^2)$ on $\tor_S(M,N)$ is independent of the choice of homotopies.

We note that $\wedge_{R}(I/I^{2}) = \wedge \tor^{S}_1(R,R)$, and one can  see the above action on $\tor$ as being, up to sign, induced by the action of the algebra $\tor^{S}(R,R)$ on $\tor^{S}(M,N)$. This is a special case of the natural product $\tor^{S}(A,B)\otimes \tor^{S}(M,N) \to \tor^{S}(A\otimes B, M\otimes N)$ defined in the book of Cartan and Eilenberg \cite[Chapter XI, Section 1]{CE},
as one can prove from the fact that  homotopies on a tensor product complex can be defined from  homotopies
on one factor. In particular, the $E$-module structure on $\tor^{S}(M,N)$ computed from homotopies on a resolution of $M$ is, up to sign,
the same as that computed from a resolution of $N$. We will not use these facts.

In this paper we focus on the structure of $\tor^{S}(M,k)$, where 
$S$ is a regular  local ring with residue field $k$, and $M$ is annihilated by $\ff c\in S$. Since $\tor^{S}(M,k)$ is annihilated by the maximal ideal, it may be regarded as a graded module over the exterior algebra
 $$E:= \wedge_{k}(I/\gm I) =: k\langle e_{1}, \dots e_{c}\rangle\,.$$

\section{Layered resolutions} \label{Sreview}\label{layered res}

Continuing with the notation of the introduction, we consider a regular local ring $S$  and a  complete intersection $R = S/(\ff c)$  of codimension $c$.  Throughout the paper all the modules are assumed  finitely generated. 

Let $M$ be a Cohen-Macaulay $S$-module of codimension $c$
and let ${\bf f}:=\ff c$ be a regular sequence in the annihilator of $M$. 
In~\cite{EP1} we construct an $S$-free resolution $\lup^S(M,\f)$
and an $R$-free resolution $\ldown_R(M,\f)$, called \emph{layered resolutions} of $M$. In this section we recall the features of the construction
 that will play a role in this paper.

The importance  of the layered resolutions comes from the following result:
\bthm \ 
\begin{enumerate}
\item If $M$ is the module of a minimal  higher matrix factorization with respect to $\ff c$ in the sense of~\cite[Definition~1.2.2]{EP}, then 
$\lup^S(M,\f)$ and $\ldown_R(M,\f)$ are minimal resolutions. 

\item
Suppose that the ground field $k$ is infinite. 
If $M$ is a sufficiently high $R$-syzygy of an $R$-module $N$ 
and the elements $f_1',\dots ,f_c'$ are sufficiently general among generators of $(\ff c )$, 
then $M$ is the module of a minimal higher matrix factorization with respect to $f_1',\dots ,f_c'$.
\end{enumerate}
\ethm

\begin{proof} See \cite[Theorems~1.3.1, 3.1.4, 5.1.2]{EP} and \cite{EP1}.
\end{proof}

\vglue .25cm
\noindent{\bf The $S$-free layered resolution.}
First we will review that  $\lup^{S}(M,\f)$ has a filtration by acyclic free subcomplexes that are  resolutions of maximal Cohen-Macaulay modules over
the intermediate rings $R(p):= S/(\ff p)$.
Let $M'(p)\to M$ be   the maximal Cohen-Macaulay $R(p)$-approximation of $M$
in the sense of \cite{AB}. We  may write
$M'(p) = M(p)\oplus R(p)^{m_p}$, where $M(p)$ has no free summand.
Following \cite[Definition 7.3.1]{EP} we call
$M(p)$  the {\it essential MCM approximation} of $M$ over $R(p)$.
Let $\L(p) := \lup^{S}(M(p), \ff p)$ be the layered $S$-free resolution of $M(p)$.
By \cite[Corollary~7.3.4]{EP} the essential Cohen-Macaulay approximation of $M(p)$ over $R(p-1)$ is $M(p-1)$, and thus we have maps 
\bequation\l{mapscm}
0=M(0)\to M(1)\to \dots \to M(c)=M\,.
\end{equation}
They induce inclusions of complexes
$$
0=\L(0)\subset \L(1)\subset \dots \subset \L(c)=:\L\,,
$$
with quotients
 \bequation
 {\bf L}(p)/{\bf L}(p-1)=\K(f_1,\dots ,f_{p-1})\otimes_S \B(p)\,,
 \l{exseqtwo} \end{equation}
where $\K(f_1,\dots ,f_{p-1})$ is the Koszul complex on $f_1,\dots ,f_{p-1}$,
and  $\B(p)$ is a two-term free complex of the form 
\bequation\l{bcomplex}
\B(p): B_1(p) \map{b_{p}} B_0(p)\,.
\end{equation}

Thus, as free $S$-modules,
 \bequation
 {\bf L}(p)={\bf L}(p-1)\oplus
S\langle  e_1,\dots , e_{p-1}\rangle\otimes_S {\bf B}(p)\, .
 \l{exseqone} \end{equation}
In particular $$A_{0}(p) := \oplus_{q=1}^pB_0(q)= \L(p)_0 \,,$$ while
$$A_{1}(p) := \oplus_{q=1}^pB_1(q)\subset \L(p)_1$$ is a summand of $\L(p)_1$.

Let 
$$
E(p) := E/(e_{p+1}, \dots, e_{c})=k\langle e_1,\dots ,e_p\rangle.
$$
From Lemma~\ref{functorial} we deduce:

\blemma\l{incl}
For $1\le p\le c-1$, the inclusion ${\bf L}(p)\subset {\bf L}$ induces an inclusion
$$\tor_S(M(p),k)\subset  \tor_S(M,k)$$  of $E(p)$-modules.
\elemma

We will make use of the following property:

\bprop\l{revan} If $M$ is the module of a minimal higher matrix factorization with respect to $ \ff c$, then
the homotopy $\sigma_p$ for $f_p$ on $\L(p)$  can be chosen so that its
component 
$$
h_p:\  \L(p)_0=\oplus_{q = 1}^p B_0(q) \to \oplus_{q = 1}^p B_1(q)\subset \L(p)_1
$$
is minimal.
\eprop

\mproof
By \cite[Theorem~5.3.1]{EP}, the  higher matrix factorization $(d,h)$ for $M$ and the homotopy $\sigma_p$ for $f_p$ can be chosen so that
$
h_p
$
 is a component of $h$.
 Furthermore, \cite[Theorem~5.1.2]{EP} shows that $R\otimes h$ is the second differential in the minimal $R$-free resolution of $M$ over $R$.
Hence, $h$ is minimal.
  \mqed

\vglue .25cm
As a complex, $\L(p)$ is a \emph{Koszul extension} of $\L(p-1)$ in the sense of \cite[Definition~3.1.1]{EP}. By definition, this is the mapping cone of a map of complexes
$$
\Psi: \K(f_1,\dots ,f_{p-1})\otimes_S \B(p)[-1] \to \L(p-1)
$$
that is zero on 
$\K(f_1,\dots ,f_{p-1})\otimes_S B_0(p)$.

\begin{example} We illustrate the constructions above in the codimension 2 case.
When $c=2$, the resolution $\lup^{S}(M,f_{1},\,f_{2})$  may be represented by the diagram:
%\mspace
\begin{center}
\begin{tikzpicture}  
 [every node/.style={scale=1},  auto]
  \node (TL) {$B_0(1)$};
  \node(TLa) [node distance=1cm, below of=TL]{$\oplus$};
    \node (TM) [node distance=4cm, left of=TL] {$B_1(1)$};
   \node(TMa) [node distance=1cm, below of=TM]{$\oplus$};
    \node (M) [node distance=2cm, below of=TM] {$B_1(2)$};
  \node(Ma) [node distance=1cm, below of=M]{$\oplus$};
       \node (L) [node distance=2cm, below of=TL] {$B_0(2)$};
       \node (BM) [node distance=2cm, below of=M] {$e_1B_0(2)$};
        \node (BL) [node distance=4cm, left of=BM] {$e_1B_1(2)$};
   \draw[->] (TM) to node  {$b_1$} (TL);
     \draw[->] (M) to node [above=10pt, left=-10pt] {$b_2$} (L);
      \draw[->] (M) to node [above=5pt, right=-20pt]  {$\psi_2$} (TL);
     \draw[->] (BL) to node [swap] {$b_2$} (BM);
  \draw[->] (BL) to node [above=12pt, right=-8pt] {$f_1$} (M);
   \draw[->] (BL) to node {} (TM);
  \draw[->] (BM) to node [above=-5pt, right=5pt]{$-f_1$} (L);
\end{tikzpicture}
  \end{center}
In the notation above $\L(1)$ is the 2-term complex 
$$
 B_{1}(1) \map{b_{1}} B_{0}(1)
$$
and 
$\K(f_{1})\otimes \B(2)$ is the complex
\begin{center}
\begin{tikzpicture}  
 [every node/.style={scale=1},  auto]
\node (L)  {$B_0(2)\,,$};
\node (M) [node distance=4cm, left of=L] {$B_1(2)$};
\node(Ma) [node distance=1cm, below of=M]{$\oplus$};
\node (BM) [node distance=2cm, below of=M] {$e_1B_0(2)$};
\node (BL) [node distance=4cm, left of=BM] {$e_1B_1(2)$};
     \draw[->] (M) to node [above=10pt, left=-10pt] {$b_2$} (L);
     \draw[->] (BL) to node [swap] {$b_2$} (BM);
  \draw[->] (BL) to node [above=12pt, right=-8pt] {$f_1$} (M);
  \draw[->] (BM) to node  [above=-5pt, right=5pt]{$-f_1$} (L);
\end{tikzpicture}
\end{center}
where $e_{1}$ denotes the basis element of $I/\gm I$ corresponding
to $f_{1}$ and we have written  three underlying free modules of
$\K(f_{1})\otimes \B(2)$, in homological degrees 2,1,0,
as
$$
  e_{1}B_{1}(2)\to
 B_{1}(2)\oplus e_{1}B_{0}(2)
\to B_{0}(2)\, .
$$ 
\end{example}

\def\T{{\bf T}}

\vglue .25cm
\noindent{\bf The $R$-free layered resolution.} 
The maps (\ref{mapscm}) induce 
inclusions of complexes
$$
0=\T (0)\subset  R\otimes \T (1)\subset \dots \subset  R\otimes \T (c)=:\T \,,
$$
where $\T(p)$ is the layered resolution of $M(p)$ over $R(p)$.
By \cite{EP1}, $\T(p+1)$ is obtained from $\T(p)$ by the Shamash construction applied to the box complex
\vglue .3cm
 \begin{center}
\begin{tikzpicture}  
 [every node/.style={scale=1},  auto]
\node(00){ };
\node(000)[node distance=3.5cm, right of=00]{$\cdots \to T(p)_{2}$};
\node(00a)[node distance=.7cm, below of=00]{};
\node(01)[node distance=3.1cm, right of=000]{$T(p)_{1}$};
\node(02)[node distance=4.5cm, right of=01]{$T(p)_{0}$};
\node(11a)[node distance=.7cm, below of=01]{$\oplus$};
\node(12a)[node distance=.7cm, below of=02]{$\oplus$};
\node(lab)[node distance=10cm, left of=12a]{$T'(p):$};
\node(11)[node distance=1.4cm, below of=01]{$R(p)\otimes  B_{1}(p)$};
\node(12)[node distance=1.4cm, below of=02]{$R(p)\otimes B_{0}(p)\,,$};
\draw[-angle 60] (000) to node  {$\partial_2$} (01);
\draw[-angle 60] (01) to node [ pos=.4]{$\partial_{1}$} (02);
\draw[-angle 60] (11) to node [below=1pt]{$b_p$} (12);
\draw[-angle 60] (11) to node [above=5pt, left=1pt] {$\psi_0$} (02);
\end{tikzpicture}
\end{center}
where $\B(p)$ is the two-term complex from (\ref{bcomplex}).
In particular, the following property holds.

\bprop\l{extprop} The CI operator $t_{j}: \T \to \T$ for $f_{j}$ on the layered resolution $\T$
can be chosen so that, for $j\le p$, it preserves the box complex ${\bf T}'(p)$ as a subcomplex of $\T$, and  its  components 
\balign
R\otimes T(p)_2 &\to R\otimes B_{0}(p)\\
R\otimes T(p)_3 &\to R\otimes B_{1}(p)\,
\end{align*}
are zero.

The dual maps $\chi_{j}: \hom(\T,R) \to \hom(\T,R)$ then
vanish on $\hom_{R}(B_{0}(p), R)$ and $\hom_{R}(B_{1}(p), R)$.
\eprop

\section{The structure of \tor }\label{S Tor Structure}

Throughout this section, as in the introduction, we  assume that $S$ is a regular local ring, $R = S/(\ff c)$ is a complete intersection of codimension $c$,
and modules are finitely generated. Write $\f := \ff c$.
We consider a module $M$ that  is the module of a minimal higher matrix factorization with respect to $\f$.
We write
$
T' := E\cdot \tor_{0}^{S}(M,k)
$
for the $E$-submodule of $\tor^{S}(M,k)$ generated by $\tor_{0}^{S}(M,k)$ and set
$T'' : =  \tor^{S}(M,k)/T'\,.$
For each module in the short exact sequence
$$
0\to T' \to \tor^{S}(M,k) \to T'' \to 0\,,
$$
 we will identify a vector space decomposition, the minimal generators as an $E$-module, a Gr\"obner basis for the relations, and the
ranks of the free modules in a minimal $E$-free resolution. In Section~\ref{Identifying the Ext Module} we will determine the structure of the resolutions themselves.

\begin{notation}\l{basic notation}
In addition to the notation and hypotheses above, we  adopt the notations $R(p), M(p), \L(p), B_{0}(p), B_{1}(p), E(p)$ of Section~\ref{layered res}. 
We set
$$
A_s(p) = \oplus_{q=1}^pB_s(q)
$$
 for $s=0,1$, as in Section~\ref{Sreview}, and $A_s:=A_s(c)$. 
We write
$\overline{-}$ for $k\otimes -$.  
\end{notation}

Since the differential
$B_1(p) \to B_0(p)$ is minimal, we may regard
$\overline{\B(p)}$ as the direct sum $\overline{B_{1}(p) \oplus B_{0}(p)}$.

The next result  is the key to the $E$-module structure of Tor:

\begin{thm}\label{decomposition}
Let the notation and hypotheses be as in~\ref{basic notation}. For $0\leq p \leq c-1$ there is an isomorphism of $E(p)$-modules, 
$$
\tor^S(M(p+1),k)  
\cong  \tor^S(M(p),k)  \oplus \biggl(E(p)\otimes_k \overline{B_0(p+1)}\biggr) \oplus \biggl(E(p)\otimes_k \overline{B_1(p+1)}\biggr)\,,
$$
where the action of $E(p)$ on the second and third summands is by
multiplication on the tensor factor $E(p)$.
\end{thm}

\mproof
The minimal free resolution $\L(p+1)$ of $M(p+1)$ is  a Koszul extension of $\L(p)$ by $\K(\ff p)\otimes_S \B(p+1)$, that is, the mapping cone of a map
$$
\psi: \K(\ff p)\otimes_S \B(p+1)[-1] \to \L(p)
$$
 such that the induced map
$\K(\ff p)\otimes_S B_0(p+1) \to \L(p)$ is zero,
as in \cite[3.1.1]{EP}. 

It follows that we may also regard $\L(p+1)$ as the  mapping cone of a  map
$$
\psi': \K(\ff p)\otimes_S B_1(p+1)[-1] \to \L(p) \oplus \K(\ff p)\otimes_S B_0(p+1)[-1].
$$
where the target complex is a direct sum, as complexes.
We can define homotopies for $\ff p$ on $\K(\ff p)\otimes_S B_0(p+1)$ and on $\K(\ff p)\otimes_S B_1(p+1)$ by 
simple multiplication on the tensor factor $\K(\ff p)$. 

We can apply Lemma~\ref{functorial} to  each of the maps in the exact sequence of the mapping cone
\begin{align*}
 0\to \L(p) \oplus (\K(\ff p)\otimes_S &B_0(p+1)) \to \L(p+1) \\
                 &\to \K(\ff p)\otimes_S B_1(p+1)\to 0\,.
\end{align*}
Thus, tensoring over $S$ with the residue field $k$, we get an exact sequence of $E(p)$-modules
\begin{align*}
 0\to \overline{\L(p)} \oplus (E(p)\otimes_{k} \overline{B_0(p+1)}) & \to \overline{\L(p+1)}\\
 & \to E(p)\otimes_{k} \overline{B_1(p+1)}\to 0\,.
\end{align*}
Since $E(p)\otimes_k \overline{B_1(p+1)} $ is a free $E(p)$-module, the sequence splits, as claimed.
\mqed

Note that Theorem~\ref{decomposition} does not assert that the ``obvious'' copy of $E(p)\otimes_k \overline{B_{1}(p+1)}$ in $\tor^{S}(M,k)$ is a submodule, but  only that there is 
a submodule isomorphic to it.

Taking $p=c$, we get:

\begin{cor} \l{decomp}\l{directs}
$$
\tor^{S}(M,k) = \bigoplus_{p=1}^{c} E(p-1)\otimes_{k} \overline{{\bf B}(p)}
$$
as vector spaces.
The subsapce $E(p-1)\otimes_{k} \overline{{\bf B}(p)}$ is an $E(p-1)$-submodule
 and the action of
 $E(p-1)$
 is via the left tensor  factor.

In particular, $\tor^{S}(M,k)$ is generated as an $E$-module in degrees
$0$ and $1$, by
 $$ \overline {A_{0}\oplus A_{1}} = \bigoplus_{p=1}^{c} \,\overline{{\bf B}(p)}\,. 
$$ 
\mqed
\end{cor}

  We can now give a Gr\"obner basis of relations for $\tor^{S}(M,k)$:

\begin{thm}\label{filtration of tor} \l{gb theorem} Suppose that $M$ is the module of a  minimal higher matrix factorization for a regular sequence $\ff c$.  Let  
$1\le p\le c$. For every 
$$
a\in \overline {\B(p)}
$$
and every $r\geq p$ there is  a homogeneous relation on $\tor^{S}(M,k)$ of the form
\bequation\l{gb}
e_{r}a  - b
\end{equation}
with 
$b\in (e_1,\dots,e_{r-1})\Big(\overline {A_{0}\oplus A_{1}}\Big)$.
These relations  form  a Gr\"obner basis 
 for the relations on $\tor^{S}(M,k)$ as an $E$-module
with respect to any term order that
refines the lexicographic order on the monomials of $E$ with $e_c\succ\dots \succ e_1$.
The $E$-module defined by the leading terms of these relations is
$$
\bigoplus_{p=1}^{c}E/(e_{p},\dots, e_{c}) \otimes_{k} \overline{{\bf B}(p)}\,.
$$

\end{thm}

\mproof
By Lemma~\ref{incl} it suffices to consider the $E(r)$-module $H(\overline{\L(r)})$.
Corollary~\ref{decomp} shows that this module can be written as
$$
\tor^S(M(r),k)\cong H(\overline{\L(r)} )= E(r-1)\cdot \biggl(\overline{A_1(r)\oplus A_0(r)}\biggr).
$$
It follows that there exists a  relation on $\tor^{S}(M,k)$ of the form
$
e_{r}a  - b
$
with $b\in E(r-1)\Big(\overline {A_{0}(r)\oplus A_{1}(r)}\Big)$.

If
$b$ had a non-zero component in  $\overline { A_0\oplus A_{1}}$, then we would have $a\in \overline{B_{0}(p)}$, which contradicts
Proposition~\ref{revan}. 
Thus $e_ra$ is the leading term of the relation in the monomial order $\succ$.
 
If we factor out these leading terms from 
the free $E$-module generated by $\overline{A_{0}\oplus A_{1}}$ we obtain
the module 
$$\bigoplus_{1\le p\le c}\, E(p-1) \otimes_k \overline{{\bf B}(p)}\,.$$
By Corollary~\ref{decomp}, this  has the same vector space dimension
as the $E$-module $\tor^{S}(M,k)$.
Therefore, the given relations form a Gr\"obner basis for the module of relations.
\mqed

We can now prove assertions (i) and (ii) of the Introduction:

\begin{thm}\l{twores} 
Suppose that $M$ is the module of a minimal higher matrix factorization for a regular sequence $\ff c$.
\begin{enumerate}
\item The submodule $T' := E\cdot \tor_{0}^{S}(M,k)$ has underlying vector space
$$
T' = \bigoplus_{p=1}^{c} E(p-1)\otimes_k \overline{B_{0}(p)}, 
$$
and thus the quotient $T'' := \tor^{S}(M,k)/T'$ has underlying vector space 
$$
T''\cong \bigoplus_{p=1}^{c} E(p-1)\otimes_k \overline{B_{1}(p)}, 
$$
where, for $s=0,1$, the action of $E(p-1)$ on the summand
$E(p-1)\otimes_k \overline{B_{s}(p)}$ is by multiplication on the left tensor factor.

\item The module $T'$ is generated by $\overline{A_{0}}$, in degree 0, while
$T''$ is generated by $\overline{A_{1}}$, in degree 1. Both $T'$ and $T''$
have linear $E$-free resolutions.

\item
The minimal free resolution of $\tor^{S}(M,k)$ as an $E$-module
is the mapping cone of a map from the minimal free resolution of $T''$ (shifted by -1)
to the minimal free resolution of $T'$.

\item The relations given in (\ref{gb}) with $a\in \overline{A_{0}}$
form a
Gr\"obner basis of the relations on the $E$-module $T'$, and those with
$a\in \overline{A_{1}}$ form a 
Gr\"obner basis of the relations on the $E$-module $T''$.
\end{enumerate}
\end{thm}

We will make use of the following well-known lemma:

\blemma\l{Tate res} The minimal $E$-free resolution of $E(p) = E/(e_{p+1},\dots,e_{c})$
 has underlying free module
$$
E\otimes_{k}\hom_{\rm gr}(k[x_{p+1}, \dots ,x_{c}], k)\,,
$$
where $k[x_{p+1}, \dots , x_{c}]$ denotes the polynomial ring
on $c-p$ variables of homological and internal degree $1$ generating a
vector space that is dual to $\hbox{\rm span}\langle e_{p+1},\dots, e_{c}\rangle$.
\elemma

\mproof The case $p=0$ is the resolution of the residue field $k$. This resolution is the ``generalized Koszul complex" of Priddy and others---see for example~\cite[Exercise~17.22]{EiBook}.

The minimal resolution of 
$E(p)$ as an $E$-module is easily seen to be the tensor product, over $k$, of $E(p)$ with the 
minimal resolution of $k$ as a module over 
the exterior algebra $k\langle e_{p+1}, \dots, e_{c}\rangle$.
\mqed

 \mproofof{Theorem}{\ref{twores}}
By Theorem \ref{decomp}, 
$$
U:= \bigoplus_{p=1}^{c} E(p-1)\otimes_k \overline {B_{0}(p)} \subseteq T'.
$$ 
The homogeneous relations~(\ref{gb}) with $a\in \overline{A_{0}}$ must have
$b\in \overline{A_{0}}$ as well, so they are relations on the free $E$-module
$E\otimes_{k}\overline A_{0}$.
If we factor out their leading terms $e_{r}a$ we obtain the $E$-module
$$
\widehat{ T'} := \bigoplus_{p=1}^c\, E/(e_{p}, \dots, e_{c})\otimes_k\overline {B_0(p)}\,.
$$
It  has the same dimension as the vector space 
$U$, proving
both  that $U =T'$ and that we have a Gr\"obner basis for $T'$.
 
As for $T''$, if we factor out the leading terms of the  relations~(\ref{gb})  with $a\in\overline{A_1}$ from 
the free $E$-module $E\otimes_{k}\overline{A_{1}}$, we obtain the $E$-module
$$
\widehat{ T''} := \bigoplus_{p=1}^c\, E/(e_{p}, \dots, e_{c})\otimes_k\overline {B_1(p)}\,,
$$
which has the same vector space dimension as $T''$, proving that these
relations form a Gr\"obner basis for the relations on $T''$ as claimed. This
concludes the proofs of parts (i) and (iv) of the Theorem.

To prove part (ii), observe first that, by Lemma~\ref{Tate res}, the minimal free resolutions  of the $E$-modules $E/(e_{p}, \dots, e_{c})$ 
are linear. It follows that the minimal $E$-free resolutions of $\widehat{ T'}$
and $\widehat{ T''}$ are linear, and thus the miminal free resolutions $\F'$ of $T'$ and $\F''$ of $T''$ are also linear.

It remains to prove part  (iii).
From the short exact sequence
$$
0\to T'\to \tor^{S}(M,k) \to T'' \to 0
$$
we see that $\tor^{S}(M,k)$ has a free resolution that is the 
mapping cone of some map of complexes $\alpha: \F''[-1]\to \F'$. 
The $j$-th term $F_{j}$ of $\F'$ is generated in degree $j$, while the $j$-th term
$F''_{j}$ of
$\F''$ is generated in degree $j+1$
since the generators $\overline {A_{1}}$ of $T''$ have degree 1.
Hence,  the matrices in the map $\alpha$
have entries of degree 2. In particular the mapping cone is a minimal resolution
of the form
$$ 
\cdots\to F'_{j}\oplus F''_{j} \to\cdots \to F'_{0}\oplus F''_{0} =  E\otimes_{k}(\overline{A_{0} \oplus
A_{1}})\,.
$$
\vglue - .5cm\ 
\mqed

In  Section~\ref{Identifying the Ext Module} we will identify the resolutions $\F'$ and $\F''$
and the map $\alpha$ in terms of the minimal free resolution of $M$ as an $R$-module. We already have enough information to interpret the Betti numbers:

\begin{thm}\l{resTor}
Suppose that $M$ is the module of a minimal higher matrix factorization for  a regular sequence $\ff c$ and for $s=1,2$ and $p = 1,\dots,c$, let $b_s(p) = \rank B_{s}(p)$. With notation as  above, the graded Betti numbers of $\tor^{S}(M,k)$ as
an $E$-module are:
\b{align*}
\beta_{i,i}^E\Big( \tor^{S}(M,k) \Big)&
=\sum_{p=1}^{c}{c-p+i\choose c-p}b_{0}(p) 
= \dim_{k}\, \ext^{2i}_{R}(M,k)\\
\beta_{i,i+1}^E\Big( \tor^{S}(M,k) \Big)& 
=\sum_{p=1}^{c}{c-p+i\choose c-p}b_{1}(p) 
= \dim_{k}\, \ext^{2i+1}_{R}(M,k)
\, ,
\end{align*}
and these two formulas give the graded Betti numbers of $T'$ and $T''$ individually.
\end{thm}

\mproof
The minimal graded $E$-free resolutions of $T'$ and $T''$  are linear,
and so their Betti numbers are equal to the Betti numbers of the modules
$\widehat T' $ and $\widehat T''$ used in the proof of Theorem \ref{twores}.
These Betti numbers can be obtained from Lemma~\ref{Tate res}.
Furthermore, the minimality of the layered resolution $\ldown_R(M,\f)$  implies the identical formula for
$\dim \,\ext_{R}^{2i}(M,k)=\beta_{2i}^R(M)$ and $\dim \,\ext_{R}^{2i+1}(M,k)=\beta_{2i+1}^R(M)$
(see \cite[Corollary~1.3.3]{EP}).
 \mqed

In experiments, we have observed that the sequence of $E$-modules
$$
0\to T' \to \tor^{S}(M,k) \to T'' \to 0\,,
$$
often splits. Here is the simplest example we know where this is not the case:

\b{example} 
Let $
 S = k[a,b,c], $ 
 $R= S/(a^{4}, b^{4}, c^{4}) $,  $$
 N = R\otimes_{S}
 \begin{coker}
\begin{pmatrix}
 a&b&c\\
 b&c&a
\end{pmatrix}
\end{coker}\, ,
$$
and let $M$ be a sufficiently high syzygy of the $R$-module $N$. Computation using Macaulay2 shows that the dual $E$-module, which up to a shift in grading is $\ext^{S}(M,k)$,
has smaller Betti numbers than does the direct sum of the duals of
$T'$ and $T''$. In particular, the $E$-submodule $T'\subset \tor_S(M,k)$ is not a direct summand.
\end{example}

\section{Regularity } \label{Sreg}

If $V$ is a finite dimensional $\mathbb Z$-graded  vector space,  we set $\max V=\max \{j\mid V_j\not=0\}$.
We define the {\it regularity} of a graded $E$-module $L$ to be
$$
\reg_{E}\,L := 
\sup_{i}\, \Big\{{\max}\,\tor^{E}_{i}(L,k)-i\,\Big\}.
$$ 
The minimal $E$-free resolution $\bf U$ of $k$ is linear, so $\tor^E(L,k)\cong {\rm H}(L\otimes {\bf U})$
 gives $\max \tor^E_i(L,k) \leq i+\max L$. Thus
$$
\reg_E(L) \leq \max L.
$$

From Theorem~\ref{twores} we get:

\bcor\l{torreg}
If  $M$ is the module of a minimal higher matrix factorization for a regular sequence $\ff c$,
then 
$$
\reg_{E}\, \tor^{S}(M,k)=1\,.$$ 
\ecor

We will provide a short alternate proof of this result. To do this, we compare
 the regularity of an
$E$-module $L$ with the regularity of $L$ regarded as a module over $E(p):=k\langle  e_1,\dots , e_{p}\rangle$, regarded as a subalgebra of $E$.

\begin{thm}\l{reg over sub} If $L$ is a finitely generated graded $E$-module, then
$$
\reg_{E}\, L \leq \reg_{E(p)}\, L \leq \reg_{E}\, L+c-p.
$$
\end{thm}

\mproof
First, we will prove the left inequality.
Take $\F$ to be the tensor product over $E(p)$ of a minimal free resolution $\G$ of $L$ as an $E(p)$-module with the minimal free resolution ${\bf D}$ of $E(p)$ as an $E$-module.  Since the latter
 is split exact as a sequence of $E(p)$-modules, $\F$ is
 a (possibly non-minimal) $E$-free resolution of $L$ as an $E$-module. 
By Lemma~\ref{Tate res},  the  resolution  ${\bf D}$
 is linear. 
Therefore, for each $i$, 
$${\max}\,F_{i}\otimes k\le {\max}_{q\le i}\{\max \, G_{q}\otimes k\}+(i-q)
\le \reg_{E(p)} L+i\,.$$

For the second inequality, note that $E = E(p)\otimes_{k} k\langle e_{p+1}, \dots, e_{c}\rangle$ is a free $E(p)$-module with
generators in degrees $\leq c-p$, so a minimal $E$-free resolution is a (possibly non-minimal)
$E(p)$-free resolution with regularity $\reg_{E}\, L+c-p$.
\mqed

We now return to the situation of Notation~\ref{basic notation}.

\begin{thm}\l{reg1}
Suppose that $M$ is the module of a minimal higher matrix factorization for a regular sequence $\ff c$ with $c\geq 1$.
The regularity of $\tor^{S}(M,k)$ as an $E$-module is $1$.
\end{thm}

\mproof 
By \cite[Theorem~3.1.4]{EP} the projective dimension of the $S$-module $M$ is $c$.
The description of its minimal resoution $\L := \lup^{S}(M,\f)$ given in Section~\ref{Sreview}
shows that the $c$-th free module in $\L$ is $E(c-1) \otimes B_{1}(c)$. Hence, $B_{1}(c)\neq 0$.

By Proposition~\ref{revan} it follows that
 the $E$-module $\tor^{S}(M,k)$ requires generators of degree 1 from $B_{1}(c)$. Thus its regularity cannot be $<1$, and we need only
prove that it is $\leq 1$. 
We will prove this by induction on $p$.

If $p=1$ then $M(1)$ is a maximal Cohen-Macaulay module over the hypersurface $S/(f_{1})$  and the resolution ${\bf L}(1)$ has projective dimension 1.  Thus,  we have $\reg_{E(1)}\,\tor^{S} (M(1),k)= 1\, $.
 
By induction on $p$, 
the direct sum in Corollary~\ref{directs} shows that 
$$\reg_{E(p-1)}\,\tor^{S} (M(p),k)\le 1\, .$$
Applying the left inequality in Theorem~\ref{reg over sub}, we conclude
 $$\reg_{E(p)}\,\tor^{S} (M(p),k)\le 1\, .$$
\vglue -.7cm \  \mqed

\section{A Gr\"obner basis for the relations on $\ext_R(M,k)$}\label{Sect: GB for ext}

As in the Introduction, we write
$ \CR:= k[\chi_{1},\dots,\chi_{c}]\,$
for the ring of CI-operators acting on   $\ext_{R}(M,k)$. Note that the $\chi_{i}$ have  degree $2$.  

Throughout this section we will suppose that $M$ is the module of a minimal higher matrix factorization for the regular sequence $f_1,\dots ,f_c$.
We will provide results for $\ext_{R}(M,k)$ as an $\CR$-module 
that are analogous to results proved above for $\tor^{S}(M,k)$ as an $E$-module.

We use the notation and hypotheses of~\ref{basic notation}.
Furthermore, we write
$-^{\vee}$ for $\hom(-,k)$.
Since the differential
$B_1(p) \to B_0(p)$ is minimal, we  may think of
${\B(p)}^{\vee}$ as the direct sum ${B_{1}(p)^{\vee} \oplus B_{0}(p)^{\vee}}$. 
We set
$$
{\cal R}(p):=k[ \chi_{p}, \dots, \chi_{c}] \subset \CR\,.
$$

The following result is the analogue of Corollary~\ref{directs}.

\bcor\l{extdecomp}
{\rm (\cite[Corollary~5.1.6]{EP})}
There is an isomorphism
 $$
\ext_{R}(M,k) \cong \bigoplus_{p=1}^{c} k[ \chi_{p}, \dots, \chi_{c}] \otimes_{k}  {\bf B}(p)^{\vee}=\bigoplus_{p=1}^{c} {\cal R}(p)\otimes_{k}  {\bf B}(p)^{\vee}
$$
of graded vector spaces. The subspace  
$$
\CR(p) \otimes  {\bf B}(p)^{\vee}
$$
is an $\CR(p)$-submodule and $\CR(p)$ acts on it via the action on the first factor.
\ecor

\vglue .25cm
The result above can be used to prove an analogue to Theorem~\ref{filtration of tor}:

\begin{thm}\label{extgb}
Suppose that $M$ is the module of a  minimal higher matrix factorization for a regular sequence $\ff c$.  Let  $1\le p\le c$. For every 
$$
a\in  {\bf B}(p)^\vee
$$
and every $r<p$ there is  a homogeneous relation on $\ext_R(M,k)$ of the form
\bequation\l{extgbeq}
\chi_{r}a  - b
\end{equation}
with $b\in (\chi_{r+1},\dots ,\chi_c)\Big(A_{0}\oplus A_{1}\Big)^{\vee}$. 
These relations  form  a Gr\"obner basis 
 for the relations on $\ext_R(M,k)$ as an $\cal R$-module
with respect to any term order that
refines the lexicographic order on the monomials of $\cal R$ with $\chi_1\succ\dots \succ \chi_c$.
The module defined by the leading terms of these relations is
$$
\bigoplus_{p=1}^{c}{\cal R}/(\chi_{1},\dots, \chi_{p-1})\, {\bf B}(p)^{\vee}\,.
$$

\end{thm}

\mproof
The existence of the desired relations follows from Proposition~\ref{extprop}.
The leading term of the relation $
\chi_{r}a  - b
$
  in the monomial order $\succ$,
is  $\chi_ra$.
 If we factor out these leading terms from 
the free $\cal R$-module generated by $(A_{0}\oplus A_{1})^{\vee}$ we obtain
the module 
$$\bigoplus_{1\le p\le c}\,  {\cal R}(p) \otimes_k\,{\bf B}(p)^{\vee}\,.$$
By Corollary~\ref{extdecomp}, this  has the same Hilbert function
as the $ {\cal R}$-module $\ext_{R}(M,k)$.
Therefore, the given relations form a Gr\"obner basis, and in particular they generate the module of all relations.
\mqed

Finally, we provide an analogue to Corollary~\ref{torreg}.

\bcor\l{extreg}
Suppose that $M$ is the module of a  minimal higher matrix factorization for a regular sequence $\ff c$.  
The $\cal R$-module $\ext_{R}^{even}(M,k)$ has regularity $0$, and the $\cal R$-module $\ext_{R}^{odd}(M,k)$ has regularity $1$.
\qed
\ecor

\section{Higher CI-operators and an inverse Eisenbud-Shamash construction}\label{inverseShamash}

The Eisenbud-Shamash Construction (see \cite{Shamash} for the codimension 1 case and \cite[Section 7]{Ei} for the general case) allows one to construct a (generally nonminimal) $R$-free resolution of an $R$-module from an $S$-free resolution together with a system of higher homotopies on the $S$-free resolution. In this section we will explain a construction that goes the other way: from an $R$-free resolution of an $R$-module together with a \emph{system of higher CI-operators} $\{t_{i}\}$ as defined below, we will construct a (generally nonminimal) $S$-free resolution. 

The classic CI-operators were first defined on $\ext_{R}(M,k)$ by Gulliksen \cite{G}, and then in the form used here by Eisenbud~\cite{Ei}. The material in this section was discovered independently by Jesse Burke, and a more general version will appear in his paper~\cite{Burke}.

\def\LL{{K}}

\begin{prop}\label{t maps} Let $S$ be a commutative ring,  let $\ff c$ be a regular sequence, and let $R = S/(\ff c).$
Let 
$$
\K := \K(\ff c): \quad \cdots\map{t_{0}'^{3}}\  \wedge^{2}S^{c}\map{t_{0}'^{2}} S^{c}\map{t_{0}'^{1}} S
$$ 
be the Koszul complex resolving $R$. Let $\overline \G$ be a complex of free
$R$-modules, and suppose that   
$$
{\bf G}:\quad \cdots \map{t_{1}'} G_{p}\map{t_{1}'} G_{p-1}\map{t_{1}'}\cdots\map{t_{1}'} G_{0}
$$
is a lifting of $\overline \G$ to a sequence of maps of free $S$-modules. There exist operators 
 $$
 t_{i} = \sum_{p,q} t^{p,q}_{i}: {\bf G}\otimes {\bf \LL}\to ({\bf G}\otimes \K)[-1]
 $$ 
 that commute with the natural action of $\wedge S^{c}$ on $\K$, having components
 $$
 t^{p,q}_{i}:G_{p}\otimes \LL_{q}\to G_{p-i}\otimes \LL_{q+i-1}
 $$ 
 for
 $i,q\geq 0, p\geq i$
 and satisfying the conditions
\balign
 t_{0}^{p,q} &= 1\otimes (-1)^{p}{t'}_{0}^{q}\\
  \quad  t_{1} &= t'_{1}\otimes 1,
\end{align*}
and
 $$
 \sum_{i+j=n}t_{i}t_{j} = 0
 $$
 for all $n$.
 The maps $R\otimes t_{i}$ are determined uniquely by these conditions.
\end{prop}

The positions of the maps $t_{0},\dots, t_{3}$, for example, are shown in the following figure
where, for clarity, the upper indices are not shown and not all the maps have been labeled:
\vglue -.5cm
 \begin{center} 
\begin{tikzpicture} [scale = 0.8, everynode/.style = {scale = 0.9}]
\node (00)  {$G_{0}\otimes K_{0}$};
\node (01)  [node distance=2.5cm, left of = 00] {$G_{1}\otimes K_{0}$};
\node (02)  [node distance=2.5cm, left of = 01]{$G_{2}\otimes K_{0}$}; 
\node (03)  [node distance=2.5cm, left of = 02]{$G_{3}\otimes K_{0}$};
\node (04)  [node distance=2.5cm, left of = 03]{$G_{4}\otimes K_{0}$};
\node (10)  [node distance = 1.7cm, below of=00]{$G_{0}\otimes K_{1}$};
\node (11)  [node distance = 1.7cm, below of=01]{$G_{1}\otimes K_{1}$};
\node (12)  [node distance = 1.7cm, below of=02]{$G_{2}\otimes K_{1}$};
\node (13)  [node distance = 1.7cm, below of=03]{$G_{3}\otimes K_{1}$}; 
\node (14)  [node distance = 1.7cm, below of=04]{$G_{4}\otimes K_{1}$}; 
\node (20)  [node distance = 1.7cm, below of=10]{$G_{0}\otimes K_{2}$};
\node (21) [node distance = 1.7cm, below of=11] {$G_{1}\otimes K_{2}$};
\node (22)  [node distance = 1.7cm, below of=12]{$G_{2}\otimes K_{2}$};
\node (23)  [node distance = 1.7cm, below of=13]{$G_{3}\otimes K_{2}$};
\node (24)  [node distance = 1.7cm, below of=14]{$G_{4}\otimes K_{2}$};
\node(0) [node distance = 1.1cm, left of = 04] {$\cdots\ $};
\node(1) [node distance = 1.1cm, left of = 14] {$\cdots\ $};
\node(2) [node distance = 1.1cm, left of = 24] {$\cdots\ $};
\node (30)  [node distance = .4cm, below of=20]{$\vdots$};
\node (31) [node distance = .4cm, below of=21] {$\vdots$};
\node (32)  [node distance = .4cm, below of=22]{$\vdots$};
\node (33)  [node distance = .4cm, below of=23]{$\vdots$};
\node (34)  [node distance = .4cm, below of=24]{$\vdots$};

\draw[->] (02) to node [above=-1pt, right=19pt]{$t_{2}$}(10);
\draw[->] (03) to node [above=-1pt, right=19pt]{$t_{2}$}(11);
\draw[->] (04) to node [above=-1pt, right=19pt]{$t_{2}$}(12);
\draw[->] (12) to node [below, left=2pt]{}(20);
\draw[->](13) to node [below, left=2pt]{}(21);
\draw[->] (14) to node [below, left=2pt]{}(22);
\draw[->] (03) to node[above=-23pt, right=60pt] {$t_{3}$}(20);
\draw[->] (04) to node [above=-23pt, right=60pt] {$t_{3}$}(21);
\draw[->] (20) to node[right] {$t_{0}$}(10);
\draw[->] (10) to node[right] {$t_{0}$}(00);
\draw[->] (21) to node[right] {}(11);
\draw[->] (11) to node[right] {}(01);
\draw[->] (22) to node[right] {}(12);
\draw[->] (12) to node[right] {}(02);
\draw[->] (23) to node[right] {}(13);
\draw[->] (13) to node[right] {}(03);
\draw[->] (24) to node[right] {}(14);
\draw[->] (14) to node[right] {}(04);
\draw[->] (01) to node[above]{$t_{1}$}(00);
\draw[->] (02) to node[above]{$t_{1}$}(01);
\draw[->] (03) to node[above]{$t_{1}$}(02);
\draw[->] (04) to node[above]{$t_{1}$}(03);
\draw[->] (11) to node[above]{}(10);
\draw[->] (12) to node[above]{}(11);
\draw[->] (13) to node[above]{}(12);
\draw[->] (14) to node[above]{}(13);
\draw[->] (21) to node[above]{}(20);
\draw[->] (22) to node[above]{}(21);
\draw[->] (23) to node[above]{}(22);
\draw[->] (24) to node[above]{}(23);

\end{tikzpicture} \end{center}
\vglue .2cm 
\mproof
 We construct the $t_{n}^{p,q}$ by induction on $n$.
 The condition $\sum_{i+j=n}t_{i}t_{j}$ holds for $n=0$ because $\K$ is a complex. The condition $\sum_{i+j=n}t_{i}t_{j}$ holds for $n=1$ by our choice of signs. 
 
Thus we assume that 
$t^{p,q}_{j}$ has been defined for all $j<n$. We next construct the maps $t_{n}^{p,0}$,
and we then define  $t_{n}^{p,q}$ for $q>0$ to be the unique maps that make
 $$
 \sum_{q}t_{n}^{p,q}: G_{p}\otimes \wedge S^{c} \to G_{p-n}\otimes \wedge S^{c}
 $$
 into a map of 
 free $\wedge S^{c}$ modules.
 
 Because $t_{0}^{p,0}=0$, the desired condition for $t_{n}^{p,0}$ is
 $$
 \sum_{i+j=n \atop j>0} t_{i}^{p-j,j-1} t_{j}^{p,0} =0.
 $$
 To simplify the notation, we drop the upper indices, which are functions of $n, p$ and $j$, and write the condition as
 $$
t_{0}t_{n} =  -\sum_{i+j=n \atop i,j>0} t_{i} t_{j}.
 $$

 Since $\K$ is acyclic,  both existence of $t_{n}$ and the uniqueness of $R\otimes t_{n}$ will follow if we show that  
$$
t_{0} \sum_{i+j=n \atop i,j>0} t_{i} t_{j}= 0.
$$
Using the induction hypothesis
$$
t_{0}t_{i} = -\sum_{\ell+m=i \atop \ell>0} t_{\ell} t_{m}
$$
for $i<n$, we get
$$
t_{0} \sum_{i+j=n \atop i,j>0} t_{i} t_{j}= -\sum_{\ell +m+j= n\atop j,\ell >0} t_{\ell} t_{m}t_{j} = -\sum_{\ell>0} t_{\ell} \sum_{m+j=n-\ell\atop j>0}t_{m}t_{j}.
$$
Since $\ell>0$ we can use the induction hypothesis again, and we see that
each sum
$$
\sum_{m+j=n-\ell\atop j>0}t_{m}t_{j}
$$
is 0, yielding the desired vanishing.
\mqed

\begin{cor}\label{GK existence}
With hypothesis as in  Proposition~\ref{t maps}, the sequence
$$
\G\K:\quad \cdots \map{}\sum_{i+j= n}G_{i}\otimes_{S}K_{j} \map{T_{n}} \sum_{i+j= n-1}G_{i}\otimes_{S}\LL_{j}\map{}\dots \map{} G_{0}\otimes K_{0}
$$
with
$$
T_{n}=\begin{pmatrix}
 t_{1}^{n,0}&t_{0}^{n-1,1}&0&\dots&0\\
t_{2}^{n,0}&t_{1}^{n-1,1}&t_{0}^{n-2,2}&\dots&0\\
\vdots&\vdots&\vdots&\ddots&\vdots\\
 t_{n}^{n,0}&t_{n-1}^{n-1,1}&t_{n-2}^{n-2,2}&\dots&t_{0}^{0,n}
\end{pmatrix}.
$$
is a complex.\mqed
\end{cor}

\begin{thm}\label{GK resolution} Let $S$ be a commutative ring,  let $\ff c$ be a regular sequence, and let $R = S/(\ff c).$
Let $\K = \K(\ff c)$ be the Koszul complex. If $\G$ is a sequence of maps of 
free $S$-modules  
$$
{\bf G}:\quad \cdots \map{t_{1}'} G_{p}\map{t_{1}'} G_{p-1}\map{t_{1}'}\cdots\map{t_{1}'} G_{0}.
$$
such that $R\otimes_{S}\G$ is an $R$-free resolution of an $R$-module $M$, then
the complex ${\bf G\LL}$ of Corollary~\ref{GK existence} is an $S$-free resolution of $M$.\mqed
\end{thm}

\mproof
To see that the complex is a resolution, we note that it is filtered by the subcomplexes
involving just the $G_{i}\otimes K_{j}$ with $i\leq m$: all the $t_{i}$ except $t_{0}$ decrease the index of $G_i$, while $t_{0}$ keeps the index of $G_i$ the same. The associated graded complex is thus the direct sum of the $G_{i}\otimes \K$, with differentials $t_{0}$; that is, the direct sum of copies of the resolution $\K$ of $R$. The homology  is thus
$\Ho_{i}({\rm gr}(\G\K))= R\otimes G_{i}$. 

In the spectral sequence converging from the homology of the associated graded complex ${\rm gr}(\G\K)$ to the homology of $\G\K$, the $E_{1}$ page thus has nonzero terms
$E_{1}^{(i,0)}=R\otimes G_{i}$ in position $(i,0)$, and differentials induced by the differential of $\G\K$. But the only differential that reduces the first index by
only 1 is $t_{1}$; thus the $E_{1}$ differential is the differential of the complex $R\otimes \G$, which is a resolution of $M$.
\mqed

For any $R$-modules $M,N$ there is a spectral sequence
$$
\tor_{i}^{R}(\tor_{j}^{S}(M,R), N) \Rightarrow \tor_{i+j}^{S}(M,N).
$$
that comes from the double complex $\G\otimes_{S}(\K\otimes_{S}N)$, and that
 allows the computation of a certain associated graded module of
$\Tor^{S}(M,N)$.
It is natural to expect that the $t_{i}$ are special liftings of the differentials in this spectral sequence;
Burke~\cite{Burke} shows that this is indeed the case. 
The complex $\G\K$ allows  the computation of
$\Tor^{S}(M,N)$ itself. 

\section{The Bernstein-Gel'fand-Gel'fand correspondence (BGG)} \label{BGG}
\def\sym{{\rm Sym}}

The results in the next section depend on properties of the 
Bernstein-Gel'fand-Gel'fand correspondence (BGG)  from \cite{EFS}, and in this short section we review what is necessary.

Let $W$ be the vector space generated by the 
regular sequence $\ff c$ so that $E = \wedge W$. We set $W = V^\vee$ and $\CR := \sym(W)$. In this section, for simplicity, we regard both $V$ and $W$ as having degree 1, though in the next section we will need to adjust to the situation where $W$ has degree 2.

The BGG correspondence establishes equivalences between the category of $\mathbb Z$-graded $\CR$-modules and linear free complexes over $E$, and also between $\mathbb Z$-graded $E$-modules and linear free complexes over $\CR$. 

For example, giving a graded vector space
$U = \oplus_{i}U_{i}$ the structure of a graded $\CR$-module is the same as giving multiplication maps $\mu_{i}: W\otimes_{k}U_{i}\to U_{i+1}$ that satisfy the commutativity and associativity conditions. But giving the map
$\mu_{i}$ is equivalent to giving a map $\delta_{i}: U_{i}\to \hom_{k}(V,U_{i+1})$, and this is equivalent, in turn, to giving a linear map of free $E$-modules
$$\hom_{k}(E, U_{i})(-1) \to \hom_{k}(E,U_{i+1})\,.$$ It turns out that the associative and commutative conditions on the $\mu_{i}$ are equivalent to the conditions
$\delta_{i+1}\delta_{i} = 0$ for all $i$. We write ${\mathbb R}(U)$ for the resulting
linear $E$-free complex with $i$-th term $\hom(E,U_{i})(i)$.

Similarly, given a graded $E$-module $T = \oplus T_{i}$ we construct a linear
$\CR$-free complex ${\mathbb L}(T)$ having $i$-the term $(\CR\otimes T_{i})(i)$. Here are the results we need:

\begin{thm} {\bf(\cite[Theorem 3.7 (Reciprocity)]{EFS})}\label{reciprocity} Let $U,T$ be finitely generated graded modules over
$\CR$ and $E$, respectively. 
  The complex 
$\mathbb L(T)$ is a free resolution of $U$ if and only if the complex
${\mathbb R}(U)$ is an injective resolution of $T$. 
\end{thm}

Since the complex $\mathbb L(T)$ is linear, the equivalent conditions of the Theorem can only be satisfied
if the Castelnuovo-Mumford regularity  of $U$ is 0. In fact this is sufficient:

\begin{cor}{\bf(\cite[Corollary 2.4]{EFS})} \label{BGG-regularity} Suppose that $U$ is a finitely generated graded $\CR$-module.
 The complex ${\mathbb R}(U)$  is acyclic
---that is, the only homology of ${\mathbb R}(U)$ is $\Ho^{0}$---if and only if
$U$ has Castelnuovo-Mumford regularity 0.
\end{cor}

\section{Free resolutions of $\tor^{S}(M,k)$ and $\ext_{R}(M,k)$}\label{Identifying the Ext Module}

We will make use of the BGG correspondence in two ways: first,
if $\F$ is any free complex of $R = S/(\ff c)$-modules, where $\ff c$ is a regular sequence, then the CI operators on $\F$ define 
an $\CR$-module structure on $H_{*}(\F \otimes_{R}k)$, and thus, since the CI operators have degree 2,
we get a linear free complex of $E$ modules
$$
 \cdots \map{t_{2}} \Ho_{2i+s}(\F\otimes_{R} k)\otimes_{k}E \map{t_{2}} \Ho_{2i+s-2}(\F\otimes_{R} k)\otimes_{k}E \map{t_{2}} \cdots.
$$
for $s = 0$ and for $s=1$. When $M$ is a high $R$-syzygy and $\F$ is its minimal free resolution,
we shall see that this is a resolution of $\tor^{S}(M,k)$.

Second, given an $S$-module $M$, the action of $E$ on the sub and quotient modules $T'$ and $T''$ of $\tor^{S}(M,k)$ 
gives us two linear complexes of $\CR$ modules; when $M$ is a high $R$-syzygy, we shall see that these
are minimal $\CR$-free resolutions of $\ext^{even}_{R}(M,k)$ and $\ext^{odd}_{R}(M,k)$, respectively.

To prove these results, we will use the complexes constructed in Corollary~\ref{GK existence}.
With notation and hypothesis as in Proposition~\ref{t maps}, we may regard $t^{0,*}_{2}$ as a map $\G \to \G\otimes S^{c}$ whose
components $t_{2,i}$  satisfy $\sum_{i}f_{i}t_{2,i} = t_{1}^2$; that is, the $R\otimes t_{2,i}$ are the same as the CI-operators defined in 
\cite{Ei}. 

\begin{cor}\l{two complexes}
With hypotheses as in Corollary~\ref{GK existence}, suppose that $R\otimes \G$ is a minimal complex. The induced maps
\begin{align*}
 t_{2}: G_{i+2}\to &G_{i}\otimes k^{c}\\
 t_{3}: G_{i+3}\to &G_{i}\otimes \wedge^{2 }k^{c}
\end{align*}
yield a complex of the form:
\vglue .5cm
 \begin{center} 
\begin{tikzpicture} [scale = 0.8, everynode/.style = {scale = 0.9}]
\node (DE)  {$\dots$}; 
\node (4)[node distance=2.7cm, right of=DE]       {$G_{4}\otimes  E$}; 
\node (2)   [node distance=4cm, right of=4]      {$ G_{2}\otimes  E$}; 
\node (0)  [node distance=4cm, right of=2]       {$G_{0}\otimes E$}; 

\node (OD)  [node distance=1cm, below of=DE]   {$\dots$}; 
\node (O1)    [node distance=1cm, below of=4]      {$\bigoplus$}; 
\node (O2)    [node distance=1cm, below of=2]      {$\bigoplus$}; 
\node (O3)     [node distance=1cm, below of=0]     {$\bigoplus$};  

\node (DO)  [node distance=1cm, below of=OD]  {$\dots$}; 
\node (5)    [node distance=1cm, below of=O1]      {$G_{5}\otimes E$}; 
\node (3)    [node distance=1cm, below of=O2]      {$G_{3}\otimes E$}; 
\node (1)      [node distance=1cm, below of=O3]     {$G_{1}\otimes E\,.$}; 

\draw[->] (DE) to node [above=10pt, right=-8pt]  {$t_{2} $} (4);
\draw[->] (4) to node [above=10pt, right=-8pt]  {$t_{2} $} (2);
\draw[->] (2) to node [above=10pt, right=-8pt]  {$t_{2} $} (0);

\draw[->] (DO) to node [above=10pt, right=-8pt]  {$t_{2} $} (5);
\draw[->] (5) to node [above=10pt, right=-8pt]  {$t_{2} $} (3);
\draw[->] (3) to node [above=10pt, right=-8pt]  {$t_{2} $} (1);

\draw[->] (5) to node [above=6pt, right=-15pt]  {$t_{3} $} (2);
\draw[->] (3) to node [above=6pt, right=-15pt]  {$t_{3} $} (0);

\end{tikzpicture} \end{center}
\end{cor}

\mproof
By minimality, $t_{1}\otimes k=0$, so $\Ho_{i}\G = G_{i}\otimes k$. Thus by Corollary~\ref{two complexes} each row of the diagram is a complex.
Further, Proposition~\ref{t maps} gives the identity $\sum_{i=0}^{5}t_{i}t_{5-i} = 0$, and tensoring with $k$ we get
$$
(t_{2}t_{3}+t_{3}t_{2}) \otimes k = 0
$$
as required.
\mqed

Note that, if $R\otimes \G$ is the minimal resolution of an $R$-module $M$, then $G_{i}\otimes E = \tor^{R}_{i}(M,k)$.

\begin{thm}\label{main7} Let $\ff c \subset S$ be a regular sequence in a regular local ring with maximal ideal $\mm$
and residue field $k$. Let $I = (\ff c)$ and let $R = S/I$.
Let $\CR: =  k[\chi_{1},\dots, \chi_{c}]$ be the ring of CI operators.
If $\reg\ \ext^{\even}_{R}(M,k) =0$ and $\reg\ \ext^{\odd}_{R}(M,k) =1$  as $\CR$-modules, where
$\chi_{i}$ acts on $\ext_{R}(M,k) = \hom_{k}(\tor^{R}(M,k),k)$ via the action of $t_{2,i}$ on $\tor^{R}(M,k)$,  then
the complex ${\bf T}(M)$:
 \begin{center} 
\begin{tikzpicture} [scale = 0.8, everynode/.style = {scale = 0.9}]
\node (DE)  {$\dots$}; 
\node (4)[node distance=2.7cm, right of=DE]       {$\Tor^{R}_{4}(M,k)\otimes_{} E$}; 
\node (2)   [node distance=4cm, right of=4]      {$ \Tor^{R}_{2}(M,k)\otimes_{} E$}; 
\node (0)  [node distance=4cm, right of=2]       {$\Tor^{R}_{0}(M,k)\otimes_{}E$};

\node (O1)    [node distance=1cm, below of=4]      {$\bigoplus$}; 
\node (O2)    [node distance=1cm, below of=2]      {$\bigoplus$}; 
\node (O3)     [node distance=1cm, below of=0]     {$\bigoplus$};  

\node (DO)  [node distance=1cm, below of=OD]  {$\dots$}; 
\node (5)    [node distance=1cm, below of=O1]      {$\Tor^{R}_{5}(M,k)\otimes_{}E$}; 
\node (3)    [node distance=1cm, below of=O2]      {$\Tor^{R}_{3}(M,k)\otimes_{}E$}; 
\node (1)      [node distance=1cm, below of=O3]     {$\Tor^{R}_{1}(M,k)\otimes_{}E$}; 

\draw[->] (DE) to node [above=10pt, right=-8pt]  {$t_{2} $} (4);
\draw[->] (4) to node [above=10pt, right=-8pt]  {$t_{2} $} (2);
\draw[->] (2) to node [above=10pt, right=-8pt]  {$t_{2} $} (0);

\draw[->] (DO) to node [above=10pt, right=-8pt]  {$t_{2} $} (5);
\draw[->] (5) to node [above=10pt, right=-8pt]  {$t_{2} $} (3);
\draw[->] (3) to node [above=10pt, right=-8pt]  {$t_{2} $} (1);

\draw[->] (5) to node [above=6pt, right=-15pt]  {$t_{3} $} (2);
\draw[->] (3) to node [above=6pt, right=-15pt]  {$t_{3} $} (0);

\end{tikzpicture} \end{center}
is a minimal free resolution of $\tor^{S}(M,k)$ as a module over $E=\tor^{S}(R,k)$. Moreover, the upper row is a minimal free resolution of the submodule
$T':= E\cdot \tor^{S}_{0}(M,k)\subset \tor^{S}(M,k)$, and the lower row is a minimal free resolution of the quotient $T'' = \tor^{S}(M,k)/T'$.
\end{thm}

Note that, if $M$ is a minimal higher matrix factorization module then, by Corollary~\ref{extreg},  the $\CR$-modules $\ext^{\even}_{R}(M,k)$ and $\ext^{\odd}_{R}(M,k)$ satisfy the 
regularity hypothesis.

We think of the minimal $E$-free resolution of $\tor^{S}(M,k)$ as having two ``strands'':
the resolution of $T'$,  and the 
resolution of $T''$. 

\mspace

\mproof
 We first show that ${\bf T}(M)$ is acyclic.
By Corollary~\ref{BGG-regularity}, the complexes corresponding to $\ext^{\even}(M,k)$ and $\ext^{\odd}(M,k)$ are
acyclic. Since $\tor^{R}(M,k)$ is the graded dual of $\ext_{R}(M,k)$ and $E$ is injective as an $E$-module, the rows of the complex
in the Theorem are acyclic. The total complex ${\bf T}(M)$ is the mapping cone of the map $t_{3}$ between these complexes, so it is acyclic as well.

Now let $\G$ be a sequence of maps of free $S$-modules such that $\G\otimes R$ is a minimal $R$-free resolution of $M$. By
Theorem~\ref{GK resolution} the homology of the complex $\G\K\otimes k$ is $\tor^{S}(M,k)$. In particular, 
$\tor^{S}_{0}(M,k), \tor^{S}_{1}(M,k)$ and  $\tor^{S}_{2}(M,k)$ are the homology of the following complexes at the middle position:

\begin{center} 
\scalebox{1}{ 
\begin{tikzpicture} 
\node (-10)       {$0$}; 
\node (00)   [node distance=4cm, right of=-10]    {$\Tor^{R}_{0}(M,k)\otimes_{}E_{0}$}; 
\node (0-1)   [node distance=4cm, right of=00]    {$0\ ,$}; 
\draw[->] (-10) to node  {} (00);
\draw[->] (00) to node   {} (0-1);
\node (20)  [node distance=1.2cm, below of=-10]       {$\Tor^{R}_{2}(M,k)\otimes_{}E_{0}$}; 
\node (01)  [node distance=4cm, right of=20]       {$\Tor^{R}_{0}(M,k)\otimes_{}E_{1}$}; 
\node (10)  [node distance=1cm, below of=01]       {$\Tor^{R}_{1}(M,k)\otimes_{}E_{0}$}; 
\node (10a)  [node distance=.5cm, below of=01]       {$\bigoplus$}; 
\node (10b)  [node distance=4cm, right of=10]       {$0\ ,$}; 
\draw[->] (20) to node [above=10pt, right=-8pt] {$t_{2}$} (01);
\draw[->] (10) to node   {} (10b);
\end{tikzpicture} 
}
\end{center}

\begin{center}
\scalebox{1}{
\begin{tikzpicture}
\node (21)   [node distance=4cm, below of=20]   {$ \Tor^{R}_{2}(M,k)\otimes_{} E_{1}$}; 
\node (02)  [node distance=4cm, right of=21]       {$\Tor^{R}_{0}(M,k)\otimes_{}E_{2}$}; 
\node (30)  [node distance=1cm, below of=21]       {$\Tor^{R}_{3}(M,k)\otimes_{}E_{0}$}; 
\node (11)  [node distance=4cm, right of=30]       {$\Tor^{R}_{1}(M,k)\otimes_{}E_{1}$}; 
\node (20)  [node distance=1cm, below of=11]       {$\Tor^{R}_{2}(M,k)\otimes_{}E_{0}$}; 
\node (01)  [node distance=4cm, right of=20]       {$\Tor^{R}_{0}(M,k)\otimes_{}E_{1}$\ .}; 

\node (O2)    [node distance=.5cm, below of=21]      {$\bigoplus$}; 
\node (O3)     [node distance=.5cm, below of=02]     {$\bigoplus$};  
\node (O4)     [node distance=.5cm, below of=11]     {$\bigoplus$};  

\draw[->] (21) to node [above=7pt, right=-8pt]  {$t_{2} $} (02);
\draw[->] (30) to node [below=7pt, right=-8pt]  {$t_{2} $} (11);
\draw[->] (30) to node [above=6pt, right=-8pt]  {$t_{3} $} (02);
\draw[->] (20) to node [above=7pt, right=-8pt]  {$t_{2} $} (01);
\end{tikzpicture} 
}
\end{center}

Under the regularity hypothesis of the Theorem the rows of the diagram ${\bf T}(M)$ are exact. In particular the map
$$
\tor_{2}^{R}(M,k)\otimes E_{0} \map{t_{2}} \tor_{0}^{R}(M,k)\otimes E_{1}
$$
in the sequence for $\tor_{2}^{S}(M,k)$ above is injective. Thus $\Ho_{0}({\bf T}(M))$ coincides
with $\tor^{S}(M,k)$ in degrees $\leq 2$. Since $\tor^{S}(M,k)$ is 1-regular by Theorem~\ref{twores}, this implies that 
$\Ho_{0}({\bf T}(M))$ coincides
with $\tor^{S}(M,k)$ in all degrees.
Together with the exactness of the two strands, this proves the Theorem.
\mqed

\begin{cor}\label{first syzygy} 
With hypotheses and notation as in Theorem~\ref{twores}, let
$M_{1}$ be the first syzygy of $M$ over $R$. The degree 0 strand of the minimal resolution of $\tor^{S}(M_{1},k)$ is  equal to the degree $1$ strand of the minimal resolution of
$\tor^{S}(M,k)$;
and the degree $1$ strand of the minimal  resolution of $\tor^{S}(M_{1},k)$ is equal to the degree $0$ strand of the minimal  resolution of
$\tor^{S}(M,k)$, truncated at homological degree $1$.\mqed
\end{cor}

We now turn to the free resolution of $\ext_{R}(M,k)$. Since the 
generators of $\CR$ have degree 2, we have
$$
\ext_{R}(M,k) = \ext^{\even}_{R}(M,k)\oplus \ext^{\odd}_{R}(M,k)
$$
as $\CR$-modules. We treat only the even part in detail, as the odd part
is analogous.

\begin{thm}\l{res of ext}
With  the hypotheses of Theorem~\ref{twores}, let
$$
 \sigma'_{i}: \ E_{1}\otimes T'_{i-1}\to T'_{i}
$$
be the multiplication maps.
The $i$-th differential in the minimal $\CR$-free resolution of 
$\ext^{\rm even}_{R}(M,k)$ 
is the map
$$
 \tau'_{i}: \ T'_{i}\otimes \CR(-i) \to T'_{i-1}\otimes_{k}\CR(-i+1)
$$
whose linear part
$$
 T'_{i}\otimes_{k} \CR_{1} \to T'_{i-1}
$$
is the vector space dual of $\sigma'_{i}$.

The corresponding statement holds for $\ext^{\odd}_{R}(M,k)$ and $T''$ as well.
\end{thm}

\begin{proof} 

 By Theorem~\ref{main7}, the minimal $E$-free resolutions of $T'$ is given by the $\CR$-module structure of the even  part of $\tor^{R}(M,k)$.
Since $\omega_{E}:=\hom(E,k) \cong E(c)$ is an injective $E$-module,
the vector space dual of this resolution is the injective resolution of the $E$-modules $\hom(T',k)$. Furthermore, the differentials in this injective resolution come, via the BGG correspondence, from the module structure of the even part of the graded vector space dual of $\tor^{R}_{\even}(M,k)$, which is the $\CR$-module $\ext_{R}^{even}(M,k)$. 

By Theorem~\ref{reciprocity}, the  resolution of 
$\ext_{R}^{\even}(M,k)$ is the BGG dual
of $\hom(T',k)$. If the module structure of $T'$
is given by maps $\mu_i: E_{1}\otimes T'_{i}\to T'_{i+1}$, then the module
structure of $\hom(T',k)$ is given by maps
$$\mu'_i:\  E_{1}\otimes \hom(T'_{i+1},k)\to \hom(T'_{i},k)\,$$
and the BGG dual complex
$$
\cdots \to E\otimes \hom(T'_{i+1},k)\to \cdots \to E\otimes \hom(T'_{1},k) \to
E\otimes \hom(T'_{0},k)
$$
 is induced by the maps
$\mu''_i:  \hom(T'_{i+1},k)\to \hom(E_{1}, k) \otimes \hom(T'_{i},k)$.
Identifying  $\hom(E_{1}, k) \otimes \hom(T'_{i},k)$ with 
$\hom(E_1\otimes T_i',k)$, we see that $\mu''_i$ is, up to change of basis, the same as
$\hom(\mu_i,k)$, proving the theorem.
\end{proof}

Since $T'_1 =  \oplus_{p} E_1(p-1)\otimes B_0(p)$, the minimal $\CR$-free presentation of $\ext_R^{\even}(M,k)$, with the hypotheses in Theorem~\ref{res of ext}, can be written
as 
\balign
\CR(-1)\otimes \left(\sum_{p=1}^c \hbox{span}_k \langle e_1,\dots, e_{p-1}\rangle \otimes B_0(p)\right)& \to \CR\otimes \left(\sum_{p=1}^c B_0(p)\right)\\
&\to \ext_R^{\even}(M,k) \to 0\,,
\end{align*}
where the map is induced by the appropriate components of the homotopies.

There is an even more direct way of getting a free presentation for the even part of Ext:

\begin{cor}\label{nonminimal presentation} With the hypothesis of Theorem~\ref{twores}, the module $\ext^{\even}_{R}(M,k)$ has an
$\CR$-free presentation as the
cokernel of the map
$$
\tau: \ext^{1}_{S}(M,k)\otimes_{k}\CR(-1) \to \hom(M,k)\otimes_{k}\CR
$$
whose linear part
$$
\mu^{\vee}: \ext^{1}_{S}(M,k) \to \hom(M,k)\otimes_{k}\CR_1
$$
 is the vector space dual of the multiplication map
$$
 \mu: E_{1}\otimes_k \tor_{0}^{S}(M,k)  \to \tor_{1}^{S}(M,k)
$$
\end{cor}

\begin{proof}
With notation as above, $\tor_{0}^{S}(M,k) = T'_{0}$, and by Theorem~\ref{res of ext} the even Ext module has minimal $\CR$-free presentation as the cokernel of the map
$$
(T'_{1})^{\vee}\otimes_{k}\CR(-1) \to (\tor_{0}^{S}(M,k))^{\vee}\otimes_{k}\CR.
$$
We have 
$$
\tor_{1}^{S}(M,k) = T'_{1}\oplus T''_{1},
$$ 
so it suffices to show that the  image of
$\mu$ is contained in $T'_{1}$. This follows from Proposition~\ref{revan}.
\end{proof}

\makeatletter
\font \sectionfont = cmbx12 
\renewcommand\section{\@startsection{section}%
{1}{0pt}{-1.8\baselineskip}{.8\baselineskip}%
{\sectionfont \center}}
\makeatother

\bibliographystyle{alpha}

\begin{thebibliography}{XYZZZ}  
\vglue .2cm


\bibitem[AB]{AB} L. L. ~Avramov and R.~Buchweitz: Homological Algebra Modulo a Regular Sequence with Special Attention to Codimension Two, {\sl J. Alg.} {\bf 230} (2000) 24--67.

\bibitem[AI]{AI} L. L.~Avramov and S.~Iyengar: Cohomology over complete intersections via exterior algebras, {Triangulated categories},
{\sl London Math. Soc. Lecture Note Ser.} {\bf 375}, Cambridge Univ. Press, Cambridge, 2010, 52--75.

\bibitem[AY]{AZ} L. L.~Avramov and Z. Yang: Betti sequences over standard graded commutative algebras with two relations. Preprint.

\bibitem[Bu]{Burke} J. Burke: Higher CI-Operators, in preparation.

\bibitem[CE]{CE} H.~Cartan and S.~Eilenberg: Homological Algebra, Reprint of the 1956 original, Princeton Landmarks in Mathematics, Princeton University Press, Princeton, NJ,  1999.


 \bibitem[Da]{Da}
H. Dao: Decent intersection and Tor-rigidity for modules over local hypersurfaces, {\sl Trans. Amer. Math. Soc.} {\bf 365} (2013), 2803--2821. 

 \bibitem[Ei1]{Ei}  D.~Eisenbud: Homological algebra on a complete
intersection, with an application to group representations,
{\sl Trans. Amer. Math. Soc.}
{\bf	260 }
(1980),
35--64. 

 
 \bibitem[Ei2]{Ei2}  D.~Eisenbud:  Periodic resolutions over exterior algebras, Special issue in celebration of Claudio Procesi's 60th birthday, {\sl J. Algebra} {\bf 258} (2002), 348--361. 


 \bibitem[Ei3]{EiBook}  D.~Eisenbud:  Commutative Algebra with a View Toward Algebraic Geometry, {\sl Graduate Texts in Math.} {\bf 150}, Springer-Verlag 1995.
 
 \bibitem[EFS]{EFS} D.~Eisenbud, G.~Fl\o ystad and F.-O.~Schreyer: 
  Sheaf cohomology and free resolutions over exterior algebras, {\sl Trans. Amer. Math. Soc.} {\bf 355} (2003),  4397--4426.

\bibitem[EP1]{EP} 
 D. Eisenbud, I. Peeva:
 Minimal Free Resolutions over Complete Intersections, {\sl Springer Lecture Notes in Math. {\bf 2152}}, 2016.
 
\bibitem[EP2]{EP1}			
 D. Eisenbud, I. Peeva:
 Layered Resolutions of Cohen-Macaulay Modules, preprint.

 
 \bibitem[Fl]{Fl} 
G.  Fl{\o}ystad: Exterior algebra resolutions arising from homogeneous bundles, {\sl Math. Scand. } {\bf 94} (2004), 191--201. 

\bibitem[Gu]{G}  T. Gulliksen: A change of ring theorem with applications to Poincar\'e series and intersection multiplicity, {\sl  Math. Scand.} {\bf 34} (1974), 167--183.
 
  \bibitem[HW]{HW} 
 C. Huneke, R. Wiegand: Tensor products of modules and the rigidity of Tor, {\sl Math. Ann.} {\bf 299} (1994), 449--476. 
 
\bibitem[M2]{M2}  
Macaulay2{\thinspace}--{\thinspace}a system for computation in
  algebraic geometry and commutative algebra programmed by D.~Grayson and M.~Stillman,
\href{http://www.math.uiuc.edu/Macaulay2/}{\tt  http://www.math.uiuc.edu/Macaulay2/}

\bibitem[Sh]{Shamash} J.~Shamash:
The Poincar\'e series of a local ring, {\sl J. Alg.}(1969), 453--470. 

\end{thebibliography}

\end{document}